\documentclass[a4paper, 12pt]{article}

\usepackage[sort&compress]{natbib}
\bibpunct{(}{)}{;}{a}{}{,} 

\usepackage{amsthm, amsmath, amssymb, mathrsfs, multirow, url, subfigure}
\usepackage{graphicx} 
\usepackage{ifthen} 
\usepackage{amsfonts}
\usepackage[usenames]{color}
\usepackage{fullpage}

\theoremstyle{plain} 
\newtheorem{thm}{Theorem}
\newtheorem{cor}{Corollary}

\newtheorem{lem}{Lemma}

\theoremstyle{definition}
\newtheorem{defn}{Definition}

\theoremstyle{remark} 
\newtheorem{ex}{Example}
\newtheorem{remark}{Remark}

\newcommand{\prob}{\mathsf{P}}

\newcommand{\unif}{{\sf Unif}}
\newcommand{\nm}{{\sf N}}

\newcommand{\bin}{{\sf Bin}}
\newcommand{\ber}{{\sf Ber}}
\newcommand{\chisq}{{\sf ChiSq}}

\newcommand{\B}{\mathcal{B}} 
\newcommand{\RR}{\mathbb{R}}

\newcommand{\XX}{\mathbb{X}}

\newcommand{\FF}{\mathbb{F}}
\newcommand{\LL}{\mathbb{L}}
\newcommand{\TT}{\mathbb{T}}

\newcommand{\ZZ}{\mathbb{Z}}

\newcommand{\sL}{\mathcal{L}}

\newcommand{\K}{\mathcal{K}}

\newcommand{\eps}{\varepsilon}

\newcommand{\cred}{\mathscr{C}}

\newcommand{\lPi}{\underline{\Pi}}
\newcommand{\uPi}{\overline{\Pi}}

\newcommand{\uGamma}{\overline{\Gamma}}

\title{Asymptotic efficiency of inferential models and a possibilistic Bernstein--von Mises theorem\footnote{This is an extended version of the paper \citep{imbvm} presented at the {\em 8th International Conference on Belief Functions}, in Belfast, UK.}}
\author{Ryan Martin \; and \; Jonathan P.~Williams}
\date{\today}

\begin{document}

\maketitle

\begin{abstract}    
The inferential model (IM) framework offers an alternative to the classical probabilistic (e.g., Bayesian and fiducial) uncertainty quantification in statistical inference.  
A key distinction is that classical uncertainty quantification takes the form of precise probabilities and offers only limited large-sample validity guarantees, whereas the IM's uncertainty quantification is imprecise in such a way that exact, finite-sample valid inference is possible.  
But are the IM's imprecision and finite-sample validity compatible with statistical efficiency?  That is, can IMs be both finite-sample valid and asymptotically efficient?  This paper gives an affirmative answer to this question via a new possibilistic Bernstein--von Mises theorem that parallels a fundamental Bayesian result.  Among other things, our result shows that the IM solution is efficient in the sense that, asymptotically, its credal set is the smallest that contains the Gaussian distribution with variance equal to the Cram\'er--Rao lower bound.  Moreover, a corresponding version of this new Bernstein--von Mises theorem is presented for problems that involve the elimination of nuisance parameters, which settles an open question concerning the relative efficiency of profiling-based versus extension-based marginalization strategies.


\smallskip

\end{abstract}

\section{Introduction}
\label{S:intro}

According to \citet{efron.cd.discuss}, ``the most important unresolved problem in statistical inference is the use of Bayes theorem in the absence of prior information.''  Numerous attempts have been made to resolve the problem, including Bayesian inference with default priors \citep{jeffreys1946, bergerbernardosun2009}, fiducial inference in Fisher's original sense \citep{fisher1935a, zabell1992} and its generalizations \citep{fraser1968, hannig.review, dempster1967, dempster2008}, and the various imprecise-probabilistic proposals \citep[e.g.,][]{berger1984, walley1991, dubois2006, augustin.etal.bookchapter}.  Many of these solutions offer a large-sample result that goes roughly as follows: the associated credible sets achieve the nominal frequentist coverage probability when the sample size, $n$, is large.  Statisticians insist on methods that are reliable in the sense of tending to report ``correct inferences'' across repeated uses, at least asymptotically, so results like this are high-value.  The key result is the so-called {\em Bernstein--von Mises theorem} which says that, under certain regularity conditions, the Bayesian or fiducial posterior distribution is approximately Gaussian, centered at the maximum likelihood estimator, and is efficient in the sense that its variance equals the Cram\'er--Rao lower bound; here ``approximately'' is in the sense that the total variation distance between the two distributions is vanishing in probability as $n \to \infty$.  For details, see \citet[][Ch.~10]{vaart1998}, \citet[][Ch.~4]{ghosh-etal-book}, or \citet{hannig.review}.  Among other things, this result ensures that the Bayesian and fiducial uncertainty quantification at least approximately meets the kind of reliability requirements mentioned above when the sample size $n$ is large.  

Still, the above Bayesian and fiducial developments have not proved to be fully satisfactory, perhaps because large-sample confidence sets fail to tell the whole story.  Indeed, the false confidence theorem \citep{balch.martin.ferson.2017, martin.nonadditive, williams.fct.2018} implies that, whether $n$ is small or large, there exists {\em false} hypotheses to which the (Bayesian or fiducial) posterior distribution will tend to assign relatively high probability or {\em confidence}.  False confidence is related to other familiar-yet-mysterious phenomena such as Stein's paradox \citep{stein1959}, the non-existence of finite-length confidence intervals \citep{gleser.hwang.1987}, and that the ``confidence'' properties enjoyed by confidence distributions are not preserved under probability calculus \citep[e.g.,][]{fraser.cd.discuss, fraser2011.rejoinder, schweder.hjort.2013}.  To eschew the inherent unreliability revealed by the false confidence theorem, precise posterior probabilities must be replaced by imprecise probabilities.  The inferential model (IM) framework, including the original developments in \citet{imbasics, imbook} and the more recent generalizations in \citet{imchar, martin.partial, martin.partial2}, does just this.  Their proposal is provably valid in a sense that implies safety from false confidence and, e.g., it provides exact finite-sample confidence sets.  But one might guess that in order to be provably valid, the IM must sacrifice on efficiency---that is, the aforementioned IM confidence sets must be larger than the asymptotically efficient Bayesian or fiducial credible sets.  The examples in the aforementioned references all demonstrate, however, that, despite the imprecision and exact validity, the IM's solutions are no less (and sometimes more) efficient than their Bayesian and fiducial counterparts.  But so far there are no general theoretical results concerning the IM's efficiency.  

Towards addressing this question about efficiency, this paper---an extended version of our conference paper, \citet{imbvm}---presents a possibilistic version of the Bernstein--von Mises theorem that applies to the class of likelihood-based IMs considered in \citet{plausfn, gim, martin.partial2}.  Specifically, we show that a version of the IM's possibility contour, properly centered and scaled, can be accurately approximated by a Gaussian possibility contour with covariance matrix equal to the Cram\'er--Rao lower bound.  More precisely, when $n$ is large, the IM's credal set is the smallest that contains the Gaussian distribution with variance equal to the Cram\'er--Rao lower bound. This result confirms our conjecture that, despite the IM's inherent imprecision, which is necessary to achieve exact validity, there is no loss of efficiency asymptotically.   It also generalizes the characterization result offered in \citet{martin.isipta2023} for group transformation models: for large $n$, the Bayesian and fiducial solutions (are roughly the same and) correspond to the inner probabilistic approximation of the IM, i.e., the Bayesian/fiducial posterior distribution is the ``most diffuse'' element in the IM's asymptotic credal set.  

This result is extended to the practically important case where to-be-eliminated nuisance parameters are present, i.e., where only some feature of the full model parameter is of interest.  There are two broadly acceptable ways to carry out marginalization in this possibilistic framework: one is based on the formal extension principle, which is purely possibility-theory-motivated, and the other is based on suitable profiling, which has statistical origins.  Empirical evidence \citep[e.g.,][]{imchar, martin.partial3} strongly suggests that the profiling-based strategy is more efficient than the extension-based strategy, but so far there is no general theory available that confirms or explains this phenomenon.  The asymptotic analysis presented in Section~\ref{S:nuisance} formally settles this matter by demonstrating that, as conjectured, the profile-based strategy gives an asymptotically tighter possibility contour than the extension-based strategy.

\section{Background}
\label{S:background}

\subsection{Possibility theory}
\label{SS:possibility}

\subsubsection{Generalities}

Possibility theory is among the simplest imprecise probability theories, corresponding to so-called {\em consonant} belief structures \citep[e.g.,][Ch.~10]{shafer1976}, i.e., belief and plausibility functions with nested focal elements.  Other key references include \citet{zadeh1978}, \citet{dubois.prade.book}, and \citet{dubois2006}. This simplicity comes in exchange for certain limits on its expressiveness, but Shafer argues that this trade-off can be justified in the statistical inference problems under consideration here:
\begin{quote}
{\em 
...Specific items of evidence often can be treated as consonant, and there is at least one general type of evidence that seems well adapted to such treatment.  This is inferential evidence---the evidence for a cause that is provided by an effect.} \citep[][p.~226]{shafer1976}
\end{quote}
The simplicity of possibilistic uncertainty quantification comes from its parallels to precise probability theory.  A necessity--possibility measure pair $(\lPi, \uPi)$ that is intended to quantify uncertainty about an uncertain $Z$ in $\ZZ$ is determined by a possibility contour $\pi: \ZZ \to [0,1]$, with $\sup_{z \in \ZZ} \pi(z) = 1$, via the rules
\[ \uPi(B) = \sup_{z \in B} \pi(z) \quad \text{and} \quad \lPi(B) = 1 - \sup_{z \not\in B} \pi(z), \quad B \subseteq \ZZ. \]
So, where ordinary probability is determined by integrating a density function, possibility is determined by optimizing a contour.  The values $\{\lPi(B), \uPi(B)\}$ are often interpreted subjectively, either as (coherent) upper and lower probabilities associated with the proposition ``$Z \in B$''---i.e., as maximum buying price and minimum selling price, respectively, for gambles $\$1(Z \in B)$ in the sense of \citet{walley1991}---or as Shaferian degrees of belief.  ``Coherence'' ensures, among other things, that the possibility measure $\uPi$ determines and is determined by its associated {\em credal set} \citep[e.g.,][Ch.~4]{levi1980}, which is the set of all ordinary/precise probabilities dominated by $\uPi$ in the sense that the numerical probabilities they assign to events are no larger than the possibilities assigned to the same events by $\uPi$.  Mathematically, the credal set is given by
\[ \cred(\uPi) = \{ \prob \in \text{probs}(\ZZ): \prob(B) \leq \uPi(B) \text{ for all $B \in \B_\ZZ$} \}, \]
where $\text{probs}(\ZZ)$ is the set of all probability measures supported on $(\ZZ, \B_\ZZ)$, with $\B_\ZZ$ the Borel $\sigma$-algebra of subsets of $\ZZ$.  So, the credal set is the non-empty, closed, and convex collection of probability measures dominated by $\uPi$, and $\uPi$ can be recovered as the upper envelope $\uPi(\cdot) = \sup\{ \prob(\cdot): \prob \in \cred(\uPi)\}$.  For a possibility measure $\uPi$, there is a simple characterization of its corresponding credal set $\cred(\uPi)$ in terms of its contour $\pi$ \citep[e.g.,][]{destercke.dubois.2014, cuoso.etal.2001}:
\begin{equation}
\label{eq:credal}
\prob \in \cred(\uPi) \iff \prob\{ \pi(Z) \leq \alpha \} \leq \alpha \quad \text{for all $\alpha \in [0,1]$}. 
\end{equation}
That is, a probability $\prob$ is an element of $\cred(\uPi)$ if its tails agree with $\pi$'s tails in the sense that $\{z: \pi(z) \leq \alpha\}$ has $\prob$-probability no more than $\alpha$ for each $\alpha$.

The aforementioned limits on the possibility measure's expressiveness are the following simple constraints: $\lPi(B) > 0$ implies $\uPi(B)=1$ and $\uPi(B) < 1$ implies $\lPi(B)=0$.  This is intuitively clear from the interpretation of $(\lPi,\uPi)$ as necessity and possibility: if there is any genuine shred of evidence that implies $B$, then $B$ must be absolutely possible; and if there is any doubt that $B$ is possible, then there can be no genuine shred of evidence that implies $B$. These limitations do not affect us here, however, because meaningful statistical inferences can be made only when the $\lPi$ values are large or the $\uPi$ values are small.  We will return to this point in Section~\ref{SS:im} below when we discuss statistical inference.  

Imprecise probabilities in general, and possibility measures in particular, can be {\em extended} so that our quantification of uncertainty about, say, $Z$ gets mapped to a corresponding quantification of uncertainty about a new variable, say, $W=f(Z,\ldots)$, which is a function of $Z$ and perhaps other things.  The dependence on ``other things'' is why this is called an ``extension''---we are reaching beyond $Z$ to make judgments about quantities that are not fully determined by $Z$---but here we will focus on the case where $W=f(Z)$ is a function of $Z$ alone; this is the case of statistical relevance in Section~\ref{S:nuisance} below.  Then Zadeh's {\em extension principle} \citep[e.g.,][]{zadeh1975a} for possibility theory says that the possibility contour for $W$ is given by 
\[ \pi^f(w) = \sup_{z: f(z)=w} \pi(z), \quad w \in \mathbb{W} = f(\ZZ). \]
This is consistent with the use of optimization to carry out ``probability'' calculations in terms of the contour $\pi$.  Indeed, if we define the set $B_w=\{z: f(z)=w\}$, then 
\[ \uPi(B_w) = \sup_{z \in B_w} \pi(z) = \sup_{z: f(z)=w} \pi(z) = \pi^f(w). \]
So this mapping of our quantification of uncertainty from $Z$ to $W=f(Z)$ follows from the basic possibilistic calculus that relies on optimization of the contour function.

\subsubsection{Special case: Gaussian possibility} 

One way to elicit a possibility measure, specifically relevant to us here, is via the {\em probability-to-possibility transform} \citep[e.g.,][]{dubois.etal.2004, hose2022thesis}.  If $p$ is a probability density function, which determines a random variable $Z \sim \prob$, then the probability-to-possibility transform defines the contour $\pi$ as 
\[ \pi(z) = \prob\{ p(Z) \leq p(z) \}, \quad z \in \ZZ. \]
This defines the ``best approximation'' of $\prob$ by a possibility measure in the sense that its corresponding credal set is the smallest one that contains $\prob$; see, e.g., \citet{dubois.prade.1990}.  If $p$ is the $D$-dimensional normal density on $\ZZ$, with mean vector $\mu$ and covariance matrix $\Sigma$, i.e., 
\[ p(z) = (2\pi)^{-D/2} \, \det(\Sigma)^{-1/2} \, \exp\bigl\{ -\tfrac12 (z-\mu)^\top \Sigma^{-1} (z-\mu) \bigr\}, \quad z \in \RR^D, \]
then the probability-to-possibility transform has contour denoted as $\gamma_{\mu, \Sigma}$ and given by 
\begin{align*}
\gamma_{\mu, \Sigma}(z) & = \prob\{ p(Z) \leq p(z) \} \\
& = \prob\bigl[ \exp\{ -\tfrac12 (Z-\mu)^\top \Sigma^{-1} (Z-\mu) \} \leq \exp\{ -\tfrac12 (z-\mu)^\top \Sigma^{-1} (z-\mu) \} \bigr] \\
& = \prob\bigl\{ (Z-\mu)^\top \Sigma^{-1} (Z-\mu) \geq (z-\mu)^\top \Sigma^{-1} (z-\mu) \bigr\} \\
& = 1 - F_D\bigl\{ (z-\mu)^\top \Sigma^{-1} (z-\mu) \bigr\}, 
\end{align*}
where $F_D$ is the $\chisq(D)$ distribution function.  Note that what we are referring to here as the ``Gaussian possibility contour'' is different from the {\em Gaussian possibility distribution, with $\pi(z) = \exp\{-\frac12 (z-\mu)^\top \Sigma^{-1} (z-\mu)\}$, that is often found in the literature \citep[e.g.,][]{denoeux.fuzzy.2022}.} The corresponding Gaussian possibility measure, defined via optimization, is denoted by $\uGamma_{\mu, \Sigma}$.  The notation $\gamma$ and $\uGamma$ will be used for the standard Gaussian possibility contour and corresponding possibility measure.  A plot of the contour $z \mapsto \gamma(z)$ for $D=1$ can be seen in Figure~\ref{fig:bernoulli}(a) below.  

An interesting application of the extension principle described above is to the special case of the Gaussian possibility measure.  For reasons that will be clear in Section~\ref{S:nuisance} below, consider the case where $Z$ partitions as $Z=(Z_1, Z_2)$, where $Z_j$ is of dimension $D_j$, with $D=D_1+D_2$, and interest is in the first component, $Z_1$.  The mean vector $\mu$, the covariance matrix $\Sigma$, and its inverse $\Sigma^{-1}$ have the corresponding partitions: 
\[ \mu=(\mu_1, \mu_2), \quad \Sigma = \begin{pmatrix} \Sigma_{11} & \Sigma_{12} \\ \Sigma_{21} & \Sigma_{22} \end{pmatrix}, \quad \text{and} \quad \Sigma^{-1} = \begin{pmatrix} \Sigma^{11} & \Sigma^{12} \\ \Sigma^{21} & \Sigma^{22} \end{pmatrix}. \]
A closed-form expression for the contour function of $Z_1$ can be obtained, by differentiating $\gamma_{\mu, \Sigma}(z_1, z_2)$ with respect to $z_2$, setting equal to 0, and solving for $z_2$ in terms of $(z_1, \mu, \Sigma)$.  The steps are a bit tedious, but the result is relatively simple.  That is, the possibility contour for $Z_1$, derived from $\gamma_{\mu, \Sigma}$ based on the extension principle, is 
\begin{equation}
\label{eq:gauss.extension}
z_1 \mapsto 1 - F_D\bigl[ (z_1 - \mu_1)^\top \{\Sigma^{11} - \Sigma^{12} (\Sigma^{22})^{-1} \Sigma^{21}\} (z_1 - \mu_1) \bigr]. 
\end{equation}
The right-hand side is a genuine possibility contour, and it closely resembles a Gaussian with mean vector $\mu_1 \in \RR^{D_1}$ and $D_1 \times D_1$ covariance matrix $\{\Sigma^{11} - \Sigma^{12} (\Sigma^{22})^{-1} \Sigma^{21}\}^{-1}$.  There is, however, one subtle difference that makes it not exactly a Gaussian contour: note that the degrees of freedom in the chi-square is $D$, rather than $D_1$ as is necessary to match the Gaussian possibility form.  This dimension-mismatch will come up again in Section~\ref{S:nuisance} as it relates to elimination of nuisance parameters.  


The previous paragraph reveals an unexpected property of our ``Gaussian possibility measure,'' namely, that the marginal possibility contour for, say, $Z_1$, derived---via the extension principle---from the joint Gaussian possibility measure for $Z=(Z_1,Z_2)$ is {\em not exactly Gaussian}.  We say ``unexpected'' because anyone reading this is as sure that ``joint Gaussianity implies marginal Gaussianity'' as he/she is sure the sun will rise tomorrow.  This apparent discrepancy can be easily resolved, however, by viewing the matter from the proper perspective.  The point is that the extension-based marginalization does not take ``joint Gaussianity'' to some kind of strict ``marginal non-Gaussianity,'' the effect is more subtle.  Recall that the chi-square distribution function has the property that $\nu \mapsto F_\nu(x)$ is strictly decreasing for each $x > 0$.  From this it follows that 
\begin{align*}
\underbrace{1 - F_{D_1}\{ (z_1 - \mu_1)^\top \Upsilon^{-1} (z_1 - \mu_1) \}}_{\text{exact Gaussian possibility, } \gamma_{\mu_1, \Upsilon}(z_1)} < \underbrace{1 - F_D\{ (z_1 - \mu_1)^\top \Upsilon^{-1} (z_1 - \mu_1)\}}_{\text{extension-based marginal possibility}},
\end{align*}
where $\Upsilon = \{\Sigma^{11} - \Sigma^{12} (\Sigma^{22})^{-1} \Sigma^{21}\}^{-1}$.  Both sides are possibility contours in $z_1$, so the above dominance implies that the credal set corresponding to the possibility contour on the right-hand side, which is that obtained via application of the extension principle, contains the credal set corresponding to that on the left-hand side.  Since the left-hand side credal set contains $\nm_{D_1}(\mu_1, \Upsilon)$, it follows that the right-hand side credal set does too.  Therefore, the possibility measure for $Z_1$ arrived at by applying the extension principle to the joint Gaussian contour for $(Z_1,Z_2)$ is still ``Gaussian'' in the sense that it (strictly) contains the credal set corresponding to the Gaussian possibility $\uGamma_{\mu_1, \Upsilon}$, which is the smallest such credal sets that contains $\nm_{D_1}(\mu_1, \Upsilon)$.  In other words, the extension-based marginalization is somewhat cautious, since it allows some distributions that are more diffuse than $\nm_{D_1}(\mu_1, \Upsilon)$ into its credal set.  We think this is not so surprising, given that the extension principle is a general operation that must accommodate, among other things, projection down to any margin, not just to $Z_1$ or to $Z_2$.  In the statistical inference context described below, the extension principle ensures that statistical validity is achieved uniformly over all possible margins, so it is unrealistic to expect that it can do so ``optimally'' for every individual margin.

\subsection{Inferential models}
\label{SS:im}

The original IM constructions \citep[e.g.,][]{imbasics, imbook} used (nested) random sets and, hence, the connection to possibility theory was indirect.  A more streamlined version in \citet{martin.partial2} directly defines the IM's possibility contour using the probability-to-possibility transform; see, also, \citet{plausfn, gim}.  An advantage of this new and direct construction is that it avoids the ambiguity in specifying both a data-generating equation and a so-called ``predictive random set''---one only needs the model/likelihood function.  

The statistical model assumes that $X^n=(X_1,\ldots,X_n)$ consists of iid samples from a distribution $\prob_\Theta$ depending on an unknown/uncertain $\Theta \in \TT$ to be inferred.  The model and observed data $X^n=x^n$ together determine a likelihood function $\theta \mapsto L_{x^n}(\theta)$ and a corresponding relative likelihood function
\[ 
R(x^n,\theta) = \frac{L_{x^n}(\theta)}{\sup_\vartheta L_{x^n}(\vartheta)}.
\]
The relative likelihood itself defines a data-dependent possibility contour that has been widely studied \citep[e.g.,][]{shafer1982, wasserman1990b, denoeux2006, denoeux2014}.  Most appealing about the likelihood-based possibility contour is its shape: peak at the maximum likelihood estimator $\hat\theta_{x^n}$ and consistent with Fisher's suggested likelihood-based preference order on the parameter space.  But the relative likelihood lacks a standard scale, i.e., what constitutes a ``small'' relative likelihood depends on aspects of the individual application.  A (literally) uniform scale of interpretation across applications can easily be obtained via what \citet{martin.partial} calls ``validification,'' which is a sort of possibilistic transform.  In particular, for observed data $X^n=x^n$, the possibilistic IM's contour is
\begin{equation}
\label{eq:contour}
\pi_{x^n}(\theta) = \prob_\theta\{ R(X^n,\theta) \leq R(x^n, \theta) \}, \quad \theta \in \TT.
\end{equation}
This is the same contour obtained in \citet{plausfn, gim} based on more-or-less the original IM construction---with a so-called ``generalized association'' based on the relative likelihood and a simple nested-interval predictive random set.  It also corresponds to the p-value associated with testing the hypothesis ``$\Theta=\theta$'' using the usual likelihood ratio test statistic; but see \citet{martin.partial2} for a deeper, principled justification.  The corresponding possibility measure, obtained via optimization, is
\[ \uPi_{x^n}(H) = \sup_{\theta \in H} \pi_{x^n}(\theta), \quad H \subseteq \TT, \]
and $\lPi_{x^n}(H) = 1 - \uPi_{x^n}(H^c)$ is its necessity measure.  Aside from the Shaferian interpretation of $\lPi_{x^n}$ and $\uPi_{x^n}$ as reporting $x^n$-dependent degrees of belief and plausibility about $\Theta$, there is a behavioral interpretation in terms of prices an agent---the data analyst, say---is willing to pay for gambles $\$1(\Theta \in H)$ as $H \subseteq \TT$ varies.  More specifically, one can interpret the IM's necessity/possibility or lower/upper probabilities as  
\begin{align*}
\lPi_{x^n}(H) & = \text{supremum price the data analyst would pay to buy $\$1(\Theta \in H)$} \\
\uPi_{x^n}(H) & = \text{infimum price the data analyst would accept to sell $\$1(\Theta \in H)$}.
\end{align*}
From this, various properties enjoyed by a necessity--possibility measure pair can be given a behavioral spin in terms of rationality.  For example, from the definition of a possibility measure, it is easy to show directly that $\lPi_{x^n}(\cdot) \leq \uPi_{x^n}(\cdot)$; in terms of buying/selling prices for gambles, note that if there was some $H$ for which this bound failed, i.e., $\lPi_{x^n}(H) > \uPi_{x^n}(H)$, then another agent could buy $\$1(\Theta \in H)$ from the data analyst and then sell it back to him/her immediately for more money, creating a sure-loss for the data analyst.  More generally, for each fixed $x^n$, the necessity--possibility pair $(\lPi_{x^n}, \uPi_{x^n})$ is a coherent imprecise probability and, therefore, free of such irrationality.  Beyond that, there is a stronger form of no-sure-loss that the possibilistic IM offers, one that considers the IM as an ``updating'' of prior beliefs in light of data; this is more complicated and beyond our present scope, so the reader is referred to \citet[][Sec.~5.2.2]{martin.partial2} for details. 

Critical to the IM developments, and seemingly (but not actually) unrelated to the behavioral properties just described, is the so-called {\em validity property}:
\begin{equation}
\label{eq:valid}
\sup_{\Theta \in \TT} \prob_\Theta\bigl\{ \pi_{X^n}(\Theta) \leq \alpha \bigr\} \leq \alpha, \quad \text{for all $\alpha \in [0,1]$}. 
\end{equation}
This is precisely the universal scaling that the relative likelihood is missing on its own, which implies that ``$\pi_{x^n}(\theta) \leq \alpha$'' has the same inferential meaning/force in every application.  It is also the familiar property satisfied by p-values, proved by an application of Fisher's probability integral transform.  By maxitivity of $\uPi_{X^n}$, it is easy to see that 
\[ \pi_{X^n}(\Theta) \leq \uPi_{X^n}(H) \quad \text{for all $H \ni \Theta$}. \]
Then an immediate consequence of \eqref{eq:valid} is 
\begin{equation}
\label{eq:valid.old}
\sup_{\Theta \in H} \prob_\Theta\{ \uPi_{X^n}(H) \leq \alpha \} \leq \alpha, \quad \text{all $\alpha \in [0,1]$ and all $H \subseteq \TT$}. 
\end{equation}
This is the original version of validity in, e.g., \citet{imbasics}; since \eqref{eq:valid.old} is a consequence of \eqref{eq:valid}, the latter is sometimes referred to as {\em strong validity}.  Property \eqref{eq:valid.old} implies that the IM is safe from false confidence: the IM will {\em not} tend to assign low plausibility to true hypotheses or, equivalently, it will not tend to assign high confidence, $\lPi_{X^n}(H)$, to false hypotheses $H$ about $\Theta$.  Finally, \eqref{eq:valid} and \eqref{eq:valid.old} together have some important and familiar statistical consequences:
\begin{itemize}
\item a test that rejects the hypothesis ``$\Theta \in H$'' when $\uPi_{x^n}(H) \leq \alpha$ will control the frequentist Type~I error rate at level $\alpha$, and 
\vspace{-2mm}
\item the set $C_\alpha(x^n) = \{\theta \in \TT: \pi_{x^n}(\theta) > \alpha\}$ is a $100(1-\alpha)$\% frequentist confidence set in the sense that its coverage probability is at least $1-\alpha$.
\end{itemize}

The argument above showing that \eqref{eq:valid.old} is a consequence of \eqref{eq:valid} actually establishes more, which is why \eqref{eq:valid} is indeed stronger than \eqref{eq:valid.old}.  Instead of simply offering calibration for fixed $H$'s, the IM offers calibration {\em uniformly} over hypotheses $H$:
\begin{equation}
\label{eq:valid.uniform}
\sup_\Theta \prob_\Theta\{ \uPi_{X^n}(H) \leq \alpha \text{ for some $H$ with $H \ni \Theta$} \} \leq \alpha, \quad \text{all $\alpha \in [0,1]$}. 
\end{equation}
This uniformity implies that the possibilistic IM offers more than tests of predetermined hypotheses, which is important because scientific investigators want/need more than that.  That is, beyond basic significance tests, investigators aiming for ``scientific discoveries'' must be able to probe for other hypotheses that are or are not supported by the data, and even reliable tests of a fixed hypotheses cannot offer reliable probing in this sense \citep[e.g.,][]{mayo.book.2018}.  We refer the reader to \citet{cella.martin.probing} for more discussion of probing and the relevance of \eqref{eq:valid.uniform}.  The point we want to end on here is that the properties above---specifically \eqref{eq:valid.uniform}---are unique to IMs that have the mathematical form of a possibility measure.  So, while there are some mathematical limits to the possibilistic IM's expressiveness, as we mentioned above, it more than makes up for these limits when it comes to reliability, which is our top priority.  

Given that the IM is inherently imprecise and offers strong, finite-sample reliability guarantees, the reader would be tempted to believe that the IM gives much more conservative results in applications---e.g., wider confidence limits---compared to, say, Bayesian solutions which are precise and asymptotically efficient.  This is not the case, however, and the next section establishes an asymptotic Gaussianity/efficiency result for IMs that debunks this myth of conservatism.  At a high-level, this combination of imprecision and (asymptotic) efficiency is made possible by the aforementioned limited expressiveness of the IM's consonant belief structure; see Remark~\ref{re:credal} below.

\section{Possibilistic Bernstein--von Mises theorem}
\label{S:bvm}

\subsection{Preview}

To build some intuition, consider a simple Gaussian model where $X^n = (X_1,\ldots,X_n)$ is an iid sample from $\nm(\Theta,\sigma^2)$, where the mean $\Theta$ is unknown  but $\sigma^2$ is known.  The maximum likelihood estimator is $\hat\theta_{x^n}=\bar x$, the sample mean, and the relative likelihood is $R(x^n, \theta) = \exp\{-\tfrac{n}{\sigma^2}(\hat\theta_{x^n} - \theta)^2\}$. From here it follows by the definition of the Gaussian possibility contour $\gamma$ in Section~\ref{SS:possibility} that the exact IM possibility contour is Gaussian:
\begin{align*}
\pi_{x^n}(\theta) & = \prob_\theta\{ R(X^n, \theta) \leq R(x^n, \theta) \} = \gamma\{ n^{1/2}(\theta-\hat\theta_{x^n}) / \sigma \}, \quad \theta \in \RR.
\end{align*}
Switching to a local parametrization, namely, $\theta = \hat\theta_{x^n} + \sigma n^{-1/2} z$, gives 
\[ \check\pi_{x^n}(z) := \pi_{x^n}(\hat\theta_{x^n} + \sigma n^{-1/2} z) = \gamma(z), \quad z \in \RR, \]
i.e., the localized/standardized IM contour is exactly standard Gaussian.  Then, for a generic $H \subseteq \TT$, the corresponding possibility measure is 
\begin{align*}
\uPi_{x^n}(H) & = 
\uGamma\bigl\{ n^{1/2}(H - \hat\theta_{x^n})/\sigma \bigr\}, 
\end{align*}
where $\uGamma$ is the Gaussian possibility measure from Section~\ref{SS:possibility} and, for constants $a$ and $b$, the set $aH+b$ is defined as $aH + b = \{a\theta + b: \theta \in H\}$.

Beyond an exactly Gaussian model, there is no longer an exact correspondence between the possibilistic IM's solution and the Gaussian possibility measure.  It's a similar story in the Bayesian and (generalized) fiducial case.  But the classical, probabilistic Bernstein--von Mises theorem implies that, under certain mild/standard regularity conditions, as $n \to \infty$, the suitably centered and scaled Bayesian posterior distribution will be approximately Gaussian.  Our main result below states that a possibilistic version of this claim holds true for the IM solution reviewed in Section~\ref{SS:im} above.


\subsection{Regularity conditions}
\label{SS:regularity}

Certain regularity conditions are required in order to establish the desired asymptotic properties, but nothing more is required here than is commonly assumed in establishing the asymptotic normality and efficiency of maximum likelihood estimators.  Intuitively, these conditions imply that the log-likelihood function be smooth enough that it can be well-approximated by a quadratic function.  One common set of regularity conditions are the classical {\em Cram\'er conditions} \citep{cramer.book}, versions of which can be found in the standard texts on statistical theory, including \citet[][Theorem~3.10]{lehmann.casella.1998} and \citet[][Theorem~7.63]{schervish1995}.  Here we will adopt a more modern approach originating in \citet{lecam1956, lecam1960, lecam1970} and \citet{hajek1972}; see, also, \citet{bickel1998} and the textbook treatments of \citet{vaart1998} and \citet{keener2010}. 

To set the scene, suppose that $\{\prob_\theta: \theta \in \TT\}$ is a statistical model, i.e., a collection of probability measures on $\XX$ indexed by a parameter space $\TT$, which we will assume is an open subset of the Euclidean space $\RR^D$, for $D \geq 1$.  We will also assume that, for each $\theta$, the probability measure $\prob_\theta$ admits a density/Radon--Nikodym derivative, $x \mapsto p_\theta(x)$, relative to the dominating $\sigma$-finite measure $\nu$ on $\XX$, which in applications would typically be counting or Lebesgue measure.  Following \citet[][Ch.~2]{bickel1998}, define the (natural) logarithm and square-root density functions:
\[ \ell_\theta(x) = \log p_\theta(x) \quad \text{and} \quad s_\theta(x) = p_\theta(x)^{1/2}. \]
The ``dot'' notation, e.g., $\dot g_\theta(x)$, will be used to represent a function that behaves like the derivative of $g_\theta(x)$ with respect to $\theta$ for fixed $x$.  If the usual partial derivative of $g_\theta(x)$ with respect to $\theta$ exists for each $x$, then $\dot g_\theta(x)$ is that derivative; but suitable functions $\dot g_\theta(x)$ may exist in certain cases even when the ordinary derivative fails to exist.  Finally, let $\sL_2(\nu)$ denote the set of measurable functions on $\XX$ that are square $\nu$-integrable. 

\begin{defn}
\label{def:reg.point}
Relative to $\{\prob_\theta: \theta \in \TT\}$, with density functions $p_\theta(x)$ and square-root density functions $s_\theta(x) = p_\theta(x)^{1/2}$, an interior point $\vartheta \in \TT$ is {\em regular} if:
\begin{itemize}
\item[(a)] there exists a vector $\dot s_\vartheta(x) = \{ \dot s_{\vartheta,d}(x): d=1,\ldots,D\}$ whose coordinates $\dot s_{\vartheta,d}$ are elements of $\sL_2(\nu)$ such that 
\begin{equation}
\label{eq:dqm}
\int \bigl| s_{\vartheta + u}(x) - s_\vartheta(x) - u^\top \dot s_\vartheta(x) \bigr|^2 \, \nu(dx) = o(\|u\|^2), \quad u \to 0 \in \RR^D, 
\end{equation}
\item[(b)] the $D \times D$ matrix $\int \dot s_\vartheta(x) \, \dot s_\vartheta(x)^\top \, \nu(dx)$ is non-singular. 
\end{itemize} 
\end{defn}

The property \eqref{eq:dqm} in Definition~\ref{def:reg.point}(a) is often described as $\theta \mapsto s_\theta$ being {\em differentiable in quadratic mean} at $\vartheta$.  The essential point is that this condition does not require the square-root density to actually be differentiable at $\vartheta$, only that it be locally ``linearizable'' in an average sense, like a differentiable-at-$\vartheta$ function would be.  The classical Cram\'er conditions, on the other hand, assume (much) more than twice continuous differentiability of the log-density.  So, the condition \eqref{eq:dqm}, which does not even require existence of a first derivative, is significantly weaker than what can be found in the classical textbooks; sufficient conditions for \eqref{eq:dqm} are given in \citet[][Lemma~7.6]{vaart1998} and discussed more below.  As the reader surely can guess, Definition~\ref{def:reg.point}(b) relates to a notion of Fisher information, and we will address this connection shortly.  

\begin{defn}
\label{def:reg.model}
The model parametrized as $\{\prob_\theta: \theta \in \TT\}$ is {\em regular} if:
\begin{itemize}
\item[(a)] $\TT$ is open and all the points $\theta \in \TT$ are regular, and 
\vspace{-2mm}
\item[(b)] the maps $\theta \mapsto \dot s_{\theta,d}$ from $\TT$ to $\sL_2(\nu)$ are continuous for each $d=1,\ldots,D$.
\end{itemize} 
\end{defn}

The role played by Condition~(b) in Definition~\ref{def:reg.model} is to facilitate the strengthening of some of the familiar ``pointwise-in-$\theta$'' asymptotic results to corresponding ``uniform-in-$\theta$'' results.  On the point of uniform convergence, say that a sequence of functions $\psi_n(z)$ converges to a function $\psi(z)$ {\em locally uniformly} if $\sup_{z \in C} |\psi_n(z) - \psi(z)| \to 0$ for every compact subset $C$ of the domain; in some contexts, what we are calling locally uniform convergence is called compact convergence. 

For a regular model $\{\prob_\theta: \theta \in \TT\}$, all $\theta$ are regular and, consequently, the $D \times D$ {\em Fisher information matrix} $I_\theta = 4 \int \dot s_\theta(x) \, \dot s_\theta(x)^\top \, \nu(dx)$ is non-singular for each $\theta \in \TT$.  Next, define the {\em score function} 
\begin{equation}
\label{eq:score} \dot\ell_\theta(x) = \frac{2 \dot s_\theta(x)}{s_\theta(x)} \, 1\{s_\theta(x) > 0 \} = \frac{\dot p_\theta(x)}{p_\theta(x)} \, 1\{p_\theta(x) > 0\}, 
\end{equation}
where $\dot p_\theta(x) := 2s_\theta(x) \dot s_\theta(x)$.  For any regular $\theta$ as in Definition~\ref{def:reg.point}, it can be shown that $\int \dot\ell_\theta(x) \, \prob_\theta(dx) = 0$, and so the Fisher information $I_\theta$ defined above is equivalent to the more familiar expression as the covariance matrix of the score function: 
\[ I_\theta = \int \dot\ell_\theta(x) \, \dot\ell_\theta(x)^\top \, \prob_\theta(dx), \quad \theta \in \TT. \]
Continuous differentiability of $\theta \mapsto p_\theta(x)$ for $\nu$-almost all $x$ and a little more is sufficient for regularity.  In particular, according to \citet[][Prop.~2.1.1]{bickel1998}, the model is regular if $\TT$ is open and the following conditions hold:
\begin{itemize}
\item $\theta \mapsto p_\theta(x)$ is continuously differentiable on $\TT$ for almost all $x$ with gradient $\dot p_\theta$, 
\vspace{-2mm}
\item the score function \eqref{eq:score} satisfies $\| \dot\ell_\theta \| \in \sL_2(\prob_\theta)$ for each $\theta$, and 
\vspace{-2mm}
\item $\theta \mapsto I_\theta$ is non-singular and continuous on $\TT$. 
\end{itemize} 
The above conditions are not necessary, however.  For example, a double-exponential location model does not satisfy the above continuous differentiability condition, but it is still regular; see \citet[][Ex.~7.6]{vaart1998} or \citet[][Ex.~2]{bickel1998}.  


Finally, we list here for reference below one relevant consequence of regularity.  Define the following sequence of random variables, indexed by $\TT$, 
\begin{equation}
\label{eq:delta}
\Delta_\theta(X^n) = n^{-1/2} I_\theta^{-1} \sum_{i=1}^n \dot\ell_\theta(X_i), \quad \theta \in \TT, 
\end{equation}
where $\dot\ell_\theta$ and $I_\theta$ are the score function and Fisher information defined above, respectively.  If $X^n$ are iid $\prob_\theta$, then this is a sum of iid random variables, with mean vector 0 and covariance matrix $I_\theta^{-1}$, scaled by $n^{-1/2}$, so it follows from the central limit theorem that $\Delta_\theta(X^n) \to \nm_D(0, I_\theta^{-1})$ in distribution, as $n \to \infty$.  More than that, by Proposition~2.1.2 in \citet{bickel1998}, if $X^n$ are iid $\prob_\theta$, then 
\begin{equation}
\label{eq:score.clt.unif}
\Delta_\theta(X^n) \to \nm_D(0, I_\theta^{-1}) \quad \text{in distribution, locally uniformly in $\theta$}.
\end{equation}

\subsection{Main result: possibilistic Bernstein--von Mises}
\label{SS:main}

Under the above regularity conditions, we can now establish a possibilistic Bernstein--von Mises theorem for the IM solution presented in Section~\ref{SS:im}.  Roughly speaking, the result states that a centered and scaled version of the IM's possibility contour can, when the sample size $n$ is large, be accurately approximated by the Gaussian possibility contour.  More specifically, if $\gamma_{\mu, \Sigma}$ is the Gaussian possibility contour in Section~\ref{SS:possibility}, define the following short-hand notation:
\begin{equation}
\label{eq:bvm.limit}
\gamma_{X^n}(\theta) = \gamma_{\Theta + n^{-1/2} \Delta_\Theta(X^n), (n I_\Theta)^{-1}}(\theta), \quad \theta \in \TT, 
\end{equation}
where $\Delta_\Theta(X^n)$ is the scaled score function \eqref{eq:delta} and $I_\Theta$ is the Fisher information matrix.  This is just a Gaussian possibility contour with a mean that is close to $\Theta$ (specifically, the mean is asymptotically equivalent to $\hat\theta_{X^n}$) and covariance matrix equal to the Cram\'er--Rao lower bound.  Then our basic claim, in Theorem~\ref{thm:bvm} below, is that the possibilistic IM contour $\pi_{X^n}$ will asymptotically merge with the Gaussian contour $\gamma_{X^n}$ in a strong, uniform sense.  Different versions of the result, which the reader might find more intuitive, will be presented following the statement of the theorem and in Section~\ref{SS:remarks}.  

\begin{thm}
\label{thm:bvm}
Let $X^n = (X_1,\ldots,X_n)$ consist of iid observations from $\prob_\Theta$, where $\{\prob_\theta: \theta \in \TT\}$ is the posited model and $\Theta \in \TT$.  Suppose that the model is regular in the sense of Definition~\ref{def:reg.model}, and that there exists a function $m: \XX \to \RR$ such that 
\begin{itemize}
\item $\int m^2 \, d\prob_\theta < \infty$ for all $\theta \in \TT$, and 
\vspace{-2mm}
\item for each pair $(\theta,\vartheta)$ and each $x$, $|\ell_\theta(x) - \ell_\vartheta(x) | \leq m(x) \, \| \theta - \vartheta \|$.
\end{itemize}
In addition, suppose that the maximum likelihood estimator is consistent.  Then, for any fixed-but-arbitrary compact subset $\K \subset \TT$, 
\begin{equation}
\label{eq:bvm}
\sup_{\theta \in \K} \bigl| \pi_{X^n}(\theta) - \gamma_{X^n}(\theta) \bigr| \to 0 \quad \text{in $\prob_\Theta$-probability as $n \to \infty$}, 
\end{equation}
where $\gamma_{X^n}$ is the Gaussian possibility contour in \eqref{eq:bvm.limit}.  
\end{thm}

The proof of Theorem~\ref{thm:bvm} is presented in Appendix~\ref{SS:bvn.proof}.  Before that, however, two technical points are made below.  Then the next two subsections offer, first, more context/intuition and, second, a few numerical illustrations of the main result.   


As promised, there are equivalent ways to represent the asymptotic approximation of the possibilistic IM's contour.  One such alternative happens to be the most convenient in certain aspects of the proof in Appendix~\ref{SS:bvn.proof} below.  Define 
\begin{equation}
\label{eq:im.scaled}
\check\pi_{X^n}(z) = \pi_{X^n}( \Theta + n^{-1/2} \, z ), \quad z \in \RR^D, 
\end{equation}
to be a suitably centered and scaled version of the IM contour.  Then Theorem~\ref{thm:bvm} implies, among other things, that
\[ \bigl| \check \pi_{X^n}(z) - \gamma_{\Delta_\Theta(X^n), I_\Theta^{-1}}(z) \bigr| \to 0 \quad \text{in $\prob_\Theta$-probability, locally uniformly in $z$}, \]
where $\gamma_{\mu, \Sigma}$ is the Gaussian possibility contour in Section~\ref{SS:possibility}, and $\Delta_\Theta(X^n)$ is the scaled score function in \eqref{eq:delta}.  This version still might not be the most intuitively pleasing or easiest to understand/interpret, but there are a couple more equivalent versions.  With a slight abuse of the previous notation, let $\ell_{X^n}(\theta) = \sum_{i=1}^n \ell_\theta(X_i)$ denote the log-likelihood function.  Then suppose that $\theta \mapsto \ell_{X^n}(\theta)$ is twice differentiable for each $X^n$; recall that Theorem~\ref{thm:bvm}'s assumption of differentiability in quadratic mean does not require that the log-likelihood be differentiable even once.  Then the maximum likelihood estimator is characterized as a solution to the likelihood equation, i.e., $\dot\ell_{X^n}(\hat\theta_{X^n}) = 0$, and the following second order condition holds: $J_{X^n} := -\ddot\ell_{X^n}(\hat\theta_{X^n})$ is non-negative definite.  Twice differentiability of the log-likelihood (and more) is entailed by the Cram\'er conditions. Under these stronger conditions, the conclusion \eqref{eq:bvm} in Theorem~\ref{thm:bvm} can be replaced by 
\begin{equation}
\label{eq:cr.bound}
\sup_{z \in K} \bigl| \pi_{X^n}\{ \hat\theta_{X^n} + (nI_\Theta)^{-1/2} \, z\} - \gamma(z) \bigr| \to 0 \quad \text{in $\prob_\Theta$-probability}, 
\end{equation}
where $\gamma$ is the standard Gaussian contour, or 
\[ \sup_{z \in \K} \bigl| \pi_{X^n}( \hat\theta_{X^n} + J_{X^n}^{-1/2} \, z) - \gamma(z) \bigr| \to 0 \quad \text{in $\prob_\Theta$-probability}. \]
The result in the second display above is the special case of Theorem~\ref{thm:bvm} presented in the conference version of this paper \citep{imbvm}.  These two are equivalent to the centering/scaling in Theorem~\ref{thm:bvm} because the scaled score function $\Delta_\Theta(X^n)$ in \eqref{eq:delta} and $n^{1/2}(\hat\theta_{X^n} - \Theta)$ are asymptotically equivalent 
and, similarly, the matrices $(nI_\Theta)$ and $J_{X^n}$ are asymptotically equivalent.  An even simpler, albeit less formal, restatement of Theorem~\ref{thm:bvm} is that, when $n$ is large, 
\begin{equation}
\label{eq:cr.bound.easy}
\pi_{X^n}(\theta) \approx \gamma_{\hat\theta_{X^n}, J_{X^n}^{-1}}(\theta), \quad \text{for all $\theta$ near $\Theta$}, 
\end{equation}
Among other things, this reveals that the IM's confidence set for $\Theta$ satisfies 
\begin{align*}
\{ \theta: \pi_{x^n}(\theta) > \alpha\} & \approx \{ \theta: \gamma_{\hat\theta_{X^n}, J_{X^n}^{-1}}(\theta) > \alpha \} \\
& = \{ \theta: 1 - F_D\bigl( (\theta - \hat\theta_{X^n})^\top J_{X^n} \, (\theta - \hat\theta_{X^n}) \bigr) > \alpha \} \\
& = \{ \theta: (\theta - \hat\theta_{X^n})^\top J_{X^n} \, (\theta - \hat\theta_{X^n}) < F_D^{-1}(1-\alpha) \},
\end{align*}
where $F_D$ is the $\chisq(D)$ distribution function, which is exactly the textbook large-sample, likelihood-based, elliptical confidence set.

\subsection{Take-home message}
\label{SS:take.home}

As stated in Section~\ref{S:intro}, the most important message concerns the (asymptotic) efficiency of the possibilistic IM solution.  This is easiest to see from \eqref{eq:cr.bound.easy}: the possibilistic IM contour asymptotically resembles the probability-to-possibility transform of a Gaussian with mean vector $\hat\theta_{X^n}$ and covariance matrix $(nI_\Theta)^{-1}$, where the latter agrees with the classical Cram\'er--Rao lower bound on the variance of unbiased estimators of $\Theta$.  Since having an asymptotic variance that agrees with the Cram\'er--Rao lower bound is the classical definition of (asymptotic) {\em efficiency}, it is fair to interpret Theorem~\ref{thm:bvm} as saying that the possibilistic IM is asymptotically efficient.  While asymptotic efficiency is desirable, it certainly is not unique to the possibilistic IM: Bayesian posterior distributions and (generalized) fiducial distributions achieve a comparable version of asymptotic efficiency.  The distinguishing feature of the possibilistic IM is that its strategic introduction of imprecision (of a possibilistic form) {\em ensures validity for all $n$}. So the take-home message is that the IM's inherent imprecision, which is what allows for finite-sample validity, remarkably does not lead to any asymptotic inefficiencies: {\em like Bayes, fiducial, etc., possibilistic IMs are asymptotically efficient; but unlike Bayes, fiducial, etc., possibilistic IMs are valid in finite samples.}  This is explained in a bit more detail in Remark~\ref{re:credal} below.

\subsection{Numerical illustrations}


\begin{ex} 
Let $X^n=(X_1,\ldots,X_n)$ consists of iid $\ber(\Theta)$ random variables, with $\Theta \in \TT = [0,1]$ uncertain.  The relative likelihood, which depends on data only through the sum $S=\sum_{i=1}^n X_i$, is 
\[ R(X^n,\theta) \equiv R(S,\theta) = \Bigl( \frac{n\theta}{S} \Bigr)^S \Bigl( \frac{n-n\theta}{n-S} \Bigr)^{n-S}, \quad \theta \in [0,1], \]
with the two boundary cases settled as $R(0,\theta) = (1-\theta)^n$ and $R(n,\theta) = \theta^n$.  There is no closed-form expression for the relative-likelihood-based IM possibility contour, but it is easy to evaluate numerically.  Indeed, if $\bar x$ is the sample proportion and $g_\theta(s)$ is the $\bin(n,\theta)$ probability mass function, then 
\[ \pi_{x^n}(\theta) = \sum_{s=0}^n 1\{ R(s,\theta) \leq R(n\bar x,\theta)\} \, g_\theta(s), \quad \theta \in [0,1], \]
with $1(\cdot)$ the indicator. The Gaussian approximation, $\theta \mapsto \gamma\{ J_{x^n}^{1/2}(\theta - \hat\theta_{x^n})\}$, suggested in Section~\ref{SS:main} specializes in this Bernoulli case to 
\[
\theta \mapsto \gamma\bigg[ \frac{n^{1/2}(\theta - \bar{x})}{\{\bar{x}(1-\bar{x})\}^{1/2}} \bigg], \quad \theta \in [0,1]. 
\]
Figure~\ref{fig:bernoulli}(a) shows the exact (standardized) possibilistic IM contour for $\Theta$, at four different sample sizes, along with the limiting standard Gaussian contour. As Theorem~\ref{thm:bvm} suggests, the approximation improves as $n$ increases.
\end{ex}

\begin{ex}
Let $X^n=(X_1,\ldots,X_n)$ consist of iid ${\sf Beta}(\Theta,1)$ random variables, where $\Theta > 0$ is unknown.  Recall that the ${\sf Beta}(\theta,1)$ has {\sc pdf} and {\sc cdf} $x \mapsto \theta x^{\theta-1}$ and $x \mapsto x^\theta$, respectively, for $x \in (0,1)$.  It can be shown that the maximum likelihood estimator is $\hat\theta_{x^n} = (-\frac1n \sum_{i=1}^n \log x_i)^{-1}$ and that the relative likelihood function is $R(x^n,\theta) = g(x^n, \theta)^n \exp\{-g(x^n,\theta)\}$, with 
\[ g(x^n,\theta) = -\textstyle\sum_{i=1}^n \log x_i^\theta. \]
The above {\sc cdf} formula implies that $X_i^\Theta \sim \unif(0,1)$, which, in turn, implies that $g(X^n,\Theta) \sim {\sf Gamma}(n,1)$.  Consequently, the exact IM contour is
\[ \pi_{x^n}(\theta) = \prob\{ G^n e^{-G} \leq R(x^n,\theta)\}, \quad \text{with} \quad G \sim {\sf Gamma}(n,1). \]
There is no closed-form expression for this, but it can easily be evaluated via Monte Carlo simulation.  Since the observed Fisher information is $J_{x^n} = n/\hat\theta_{x^n}^2$, we can now readily evaluate the standardized IM contour $z \mapsto \check\pi_{x^n}(z)$ and compare to the Gaussian approximation.  Figure~\ref{fig:bernoulli}(b) shows the contours for four different values of $n$, and the Gaussian approximation improves as $n$ increases. 
\end{ex}

\begin{figure}[t]
\begin{center}
\subfigure[Example~1: Bernoulli]{\scalebox{0.55}{\includegraphics{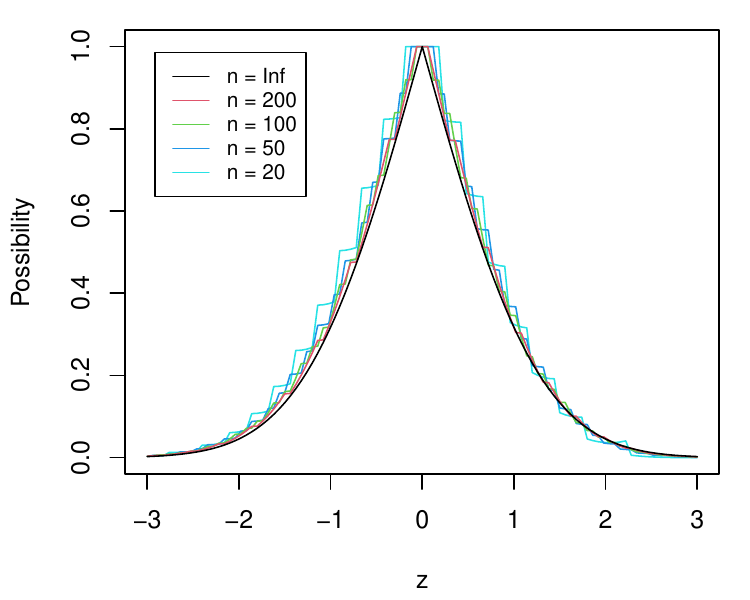}}}
\subfigure[Example~2: Beta]{\scalebox{0.55}{\includegraphics{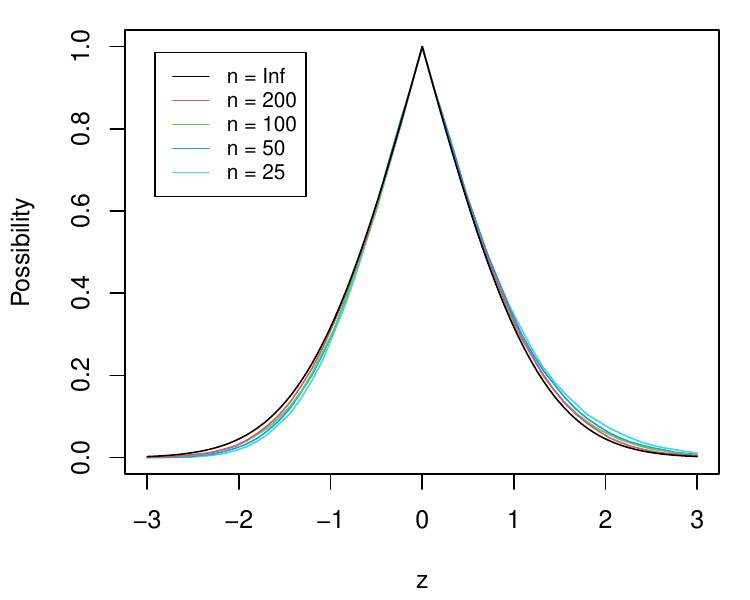}}}
\end{center}
\vspace{-4mm}
\caption{Possibility contours for the two illustrative examples.}
\label{fig:bernoulli}
\end{figure}

\begin{ex}
\label{ex:tri}
An interesting and classical non-standard example is the so-called asymmetric triangular distribution \citep[e.g.,][Example~11]{bergerbernardosun2009}, with density function
\[ p_\theta(x) = \begin{cases} 2x/\theta  & \text{if $0 \leq x \leq \theta$}, \\ 2(1-x) / (1-\theta)  & \text{if $\theta < x \leq 1$}, \end{cases} \]
where $\theta \in [0,1]$.  Suppose that data $X^n$ is iid from the triangular model with unknown mode $\Theta \in \TT = [0,1]$.  While not a common model in applications, it has a long history, going all the way back to \citet{simpson1755}; see, also, \citet{schmidt1934}, \citet{johnson.kotz.1999}, and, more recently, \citet{gim} and \citet{borgert.hannig.bvm}.  What makes this model interesting and challenging is, among other things, that, while the density is continuous and has a unique mode at $\theta$, it is not differentiable there.  This creates difficulties for the classical statistical theory that requires existence of at least two derivatives of the log-likelihood function.  But smoothness is not required to construct a possibilistic IM; in fact, as discussed in \citet{gim}, it can be done relatively easily with brute force Monte Carlo.  Figure~\ref{fig:tri} shows the (Monte Carlo-driven) exact IM contours for two data sets of different sizes, namely, $n=250$ and $n=500$, along with the large-sample Gaussian approximations.  For the Gaussian approximation, we use the formula for the Fisher information, $I_\theta = \{\theta(1-\theta)\}^{-1}$, derived in \citet{borgert.hannig.bvm}.  This model is rather complex, nearly non-regular, so there are clear signs of non-Gaussianity even at $n=250$, but these are mostly gone at $n=500$.  
\end{ex}

\begin{figure}[t]
\begin{center}
\subfigure[$n=250$]{\scalebox{0.55}{\includegraphics{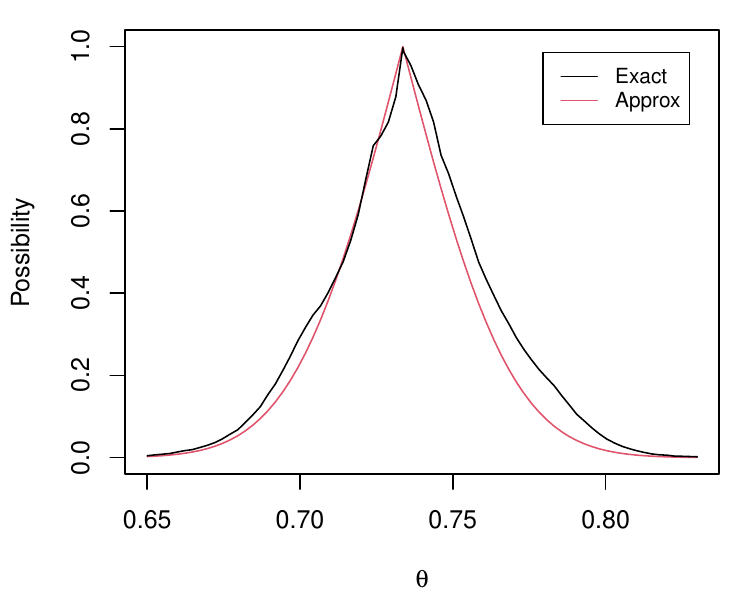}}}
\subfigure[$n=500$]{\scalebox{0.55}{\includegraphics{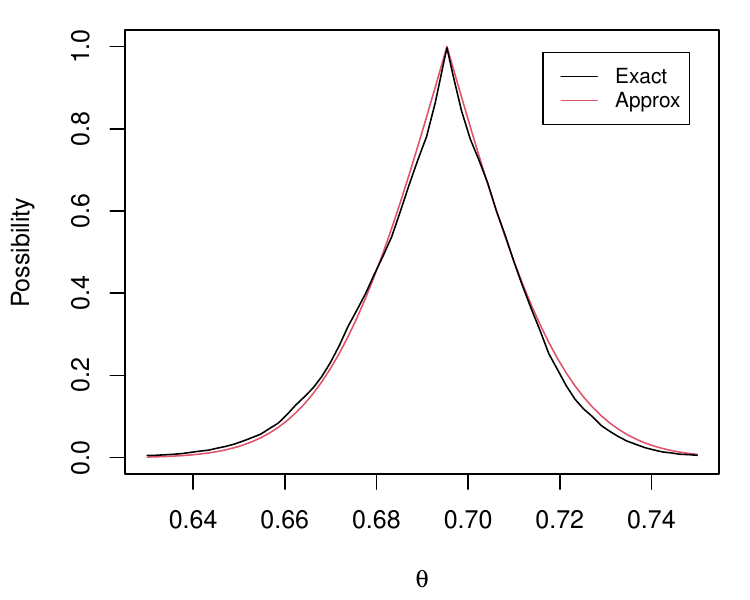}}}
\end{center}
\caption{Plot of the exact and approximate IM contour for the mode $\Theta$ in the triangular model as described in Example~\ref{ex:reg.variance}, for two different sample sizes when $\Theta=0.7$.}
\label{fig:tri}
\end{figure}

\begin{ex}
\label{ex:logistic}
Data presented in Table 8.4 of \citet[][p.~252]{ghosh-etal-book} concern the relationship between exposure to chloracetic acid and the death of mice. A total of $n = 120$ mice were exposed, ten at each of the twelve dose levels (denoted by $x$), and a binary death indicator is measured (denoted by $y$).  Figure~\ref{fig:logistic}(a) shows a plot of the data, with jitter in $x$ to show the replications, along with a simple logistic regression model fit.  Let $\Theta=(\Theta_0, \Theta_1)$ denote the uncertain value of the logistic regression model parameter, the ``intercept'' and ``slope'' pair, respectively.  

Concerning inference on $\Theta$ itself, computation of the ``exact'' IM contour requires lots of Monte Carlo evaluations and, hence, is very expensive.  The Gaussian approximation, however, is incredibly easy to evaluate, and Figure~\ref{fig:logistic}(b) shows, as Theorem~\ref{thm:bvm} predicts, that the approximation is quite accurate.  
\end{ex} 

\begin{figure}[t]
\begin{center}
\subfigure[Scatterplot, model fit, and $\hat\phi_n=0.244$]{\scalebox{0.55}{\includegraphics{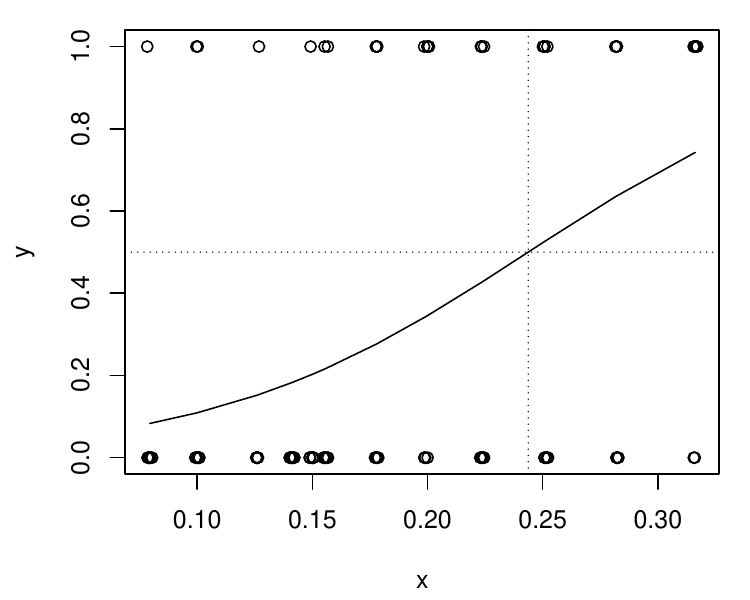}}}
\subfigure[Joint contour for $\Theta=(\Theta_0,\Theta_1)$]{\scalebox{0.55}{\includegraphics{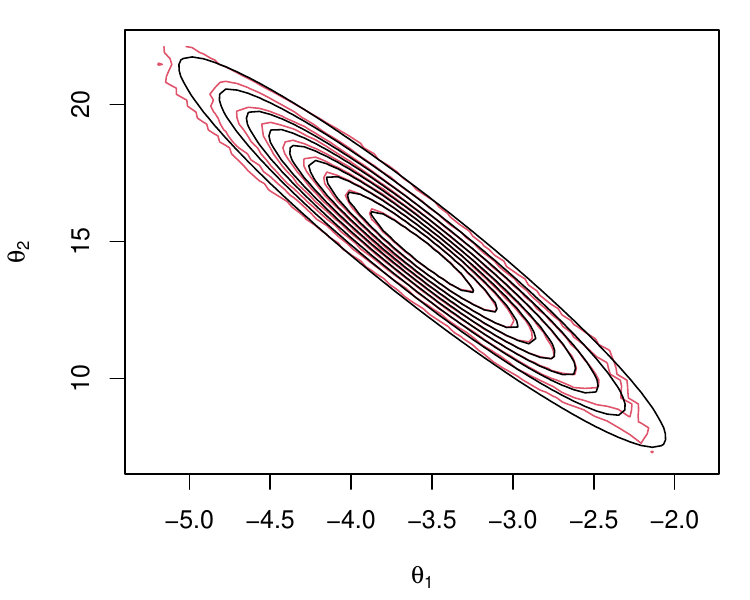}}}
\end{center}
\vspace{-4mm}
\caption{Results for logistic regression in Example~\ref{ex:logistic}. Panel~(a) shows the data (with jitter in the $x$ values for visualization) and the best-fit logistic regression line; the dotted lines will be explained in Example~\ref{ex:logistic.again}. In Panel~(b), the (rough Monte Carlo approximation of the) exact IM contour is red and the Gaussian approximation is black.}
\label{fig:logistic}
\end{figure}

\subsection{Implications for uncertainty quantification} 

The uniformity in Theorem~\ref{thm:bvm} is critical because the contour function itself is only the basic building block of the IM.  Our primary goal is reliable uncertainty quantification---safe from false confidence, etc.---which involves the corresponding necessity and possibility measures, $(\lPi_{X^n}, \uPi_{X^n})$.  So, it is important to explain the implications of Theorem~\ref{thm:bvm} as it pertains to the possibility measure $\uPi_{X^n}$ derived from the contour $\pi_{X^n}$.  The point is that uniform control on the contour in \eqref{eq:bvm} is exactly what is needed to control the possibility measure determined by optimization of the contour.  We proceed informally since the formal details are likely to over-complicate these relatively simple points.  

To start, if $f$ and $g$ are two generic, real-valued functions defined on the same domain, and if $A$ is a subset of that common domain, then it is easy to show that  
\begin{equation}
\label{eq:sups}
\Bigl| \sup_A f - \sup_A g \Bigr| \leq \sup_A |f-g|. 
\end{equation}
Then \eqref{eq:sups} implies that, for any hypothesis sequence $H_n \subseteq \TT$, 
\begin{align*}
\bigl| \uPi_{X^n}(H_n) - \uGamma_{X^n}(H_n) \bigr| & = \Bigl| \sup_{\theta \in H_n} \pi_{X^n}(\theta) - \sup_{\theta \in H_{n}} \gamma_{X^n}(\theta) \Bigr| \\
& \leq \sup_{\theta \in H_n} \bigl| \pi_{X^n}(\theta) - \gamma_{X^n}(\theta) \bigr|. 
\end{align*}
Let $\tilde\theta_{X^n}(H_n)$ be a point at which the supremum in the upper bound above is attained, which may or may not be in $H_n$.  If $H_n$ is such that $\tilde\theta_{X^n}(H_n)$ is bounded, then it follows from \eqref{eq:bvm} that 
\begin{equation}
\label{eq:bvm.H}
\bigl| \uPi_{X^n}(H_n) - \uGamma_{X^n}(H_n) \bigr| \to 0, \quad \text{in $\prob_\Theta$-probability as $n \to \infty$}.
\end{equation}
The potential for over-complication is created if we try to offer an overarching sufficient condition for ``$H_n$ is such that $\tilde\theta_{X^n}(H_n)$ is bounded.''  This is immediately clear for any static $H$ with $\Theta$ in the interior of $H$ or in the interior of $H^c$: since $\hat\theta_{X^n}$ and $\Theta + n^{-1/2} \Delta_\Theta(X^n)$ are asymptotically equivalent and, hence, both are consistent by assumption, the global modes of $\pi_{X^n}$ and $\gamma_{X^n}$ are eventually in whichever of $H$ or $H^c$ contains $\Theta$.  For static $H$ with $\Theta \in H$ but on the boundary, the respective modes will forever bounce in and out of $H$, but they are both collapsing to $\Theta$ so both $\uPi_{X^n}(H)$ and $\uGamma_{X^n}(H)$ will converge to 1.  The property \eqref{eq:bvm.H} also holds for genuine sequences of hypotheses.  It clearly holds for $H_n$ such that $\Theta$ is eventually contained in the interior of $H_n$ or of $H_n^c$, for the reasons discussed above.  It also holds for sequences $H_n$ such that $\tilde\theta_{X^n}(H_n)$ converges in $\prob_\Theta$-probability to $\Theta$: since one can view $\tilde\theta_{X^n}(H_n)$ as something like a ``projection'' of $\Theta$ onto $H_n$, we see that this holds if, roughly, $H_n$ contains points that are converging to $\Theta$.  Finally, although this is a bit unusual, the property \eqref{eq:bvm.H} also holds for certain data-dependent sequences $H_n \equiv H_{X^n}$.  We mentioned above---see the discussion following Equation~\eqref{eq:valid.uniform}---that data-dependent hypotheses can be relevant when investigators are probing for hypotheses supported by the data.  As an extreme example, consider $H_{X^n} = \{\hat\theta_{X^n}\}$.  This is not strictly a true hypothesis, but it is ``effectively, asymptotically true,'' so we expect the data to be compatible with it and, indeed, $\uPi_{X^n}(H_{X^n}) \equiv 1$.  But, again, since the mode $\Theta + n^{-1/2} \Delta_\Theta(X^n)$ of $\gamma_{X^n}$ is asymptotically equivalent to $\hat\theta_{X^n}$, we find that \eqref{eq:bvm.H} also holds for this (extreme) data-dependent $H_{X^n}$, as expected. 

One point in the above discussion that deserves further emphasis is summarized in the following corollary.  

\begin{cor}
\label{cor:consistency}
Let $H$ be a fixed hypothesis with $\Theta$ in its interior.  Under the conditions of Theorem~\ref{thm:bvm}, $\uPi_{X^n}(H) \to 1$ and $\uPi_{X^n}(H^c) \to 0$ in $\prob_\Theta$-probability. 
\end{cor}

From a statistical perspective, this result parallels the Bayesian posterior consistency theorems \citep[e.g.][]{ghosh-etal-book, ghosal.vaart.book}.  In words, it says that, for any hypothesis that can be unequivocally classified as {\em true} or as {\em false}, the possibilistic IM's uncertainty quantification will converge to that classification as $n \to \infty$.  Beyond the familiar statistical consequences, this result also has interesting epistemological implications.  Recently \citet{lin2024} 
described a {\em convergentist} tradition wherein methods of inference are evaluated based their ``convergence to truth'' properties; Lin traces this perspective back to \citet{peirce.cp} and \citet{reichenbach1938}.  So, while IMs deviate from Bayesianism in desirable ways (e.g., can elegantly accommodate the vacuous prior case and offers finite-sample validity), it is still a deviation, so one could ask what makes the IM framework philosophically sound.  A more complete answer to this question can be given (and will be given elsewhere), but Theorem~\ref{thm:bvm} and, in particular, Corollary~\ref{cor:consistency} demonstrate that IMs are at least grounded in Lin's convergentist tradition.  Circling back to the discussion of the IM output's subjective interpretation in Section~\ref{SS:im}, if an agent adopts such an interpretation, then Corollary~\ref{cor:consistency} implies that he/she will be ``right'' as $n \to \infty$.  That is, his/her subjective beliefs in true hypotheses $H$ (with $\Theta$ in the interior) converges to 1 and, similarly, he/she would be willing to buy (resp.~sell) a gamble in a true (resp.~false) $H$ for any amount less than \$1 (resp.~greater than \$0).

Corollary~\ref{cor:consistency} and the discussion in the above paragraph can be generalized considerably.  Recall that the upper probability $\uPi_{X^n}(H)$ is just a Choquet integral of the indicator function $\theta \mapsto 1(\theta \in H)$ with respect to the set function/capacity $\uPi_{X^n}$.  Then the point we want to make here is that there is nothing special about indicator functions.  But before making this point, we first need a definition. 
 
For a function $f: \TT \to \RR$, the Choquet integral with respect to $\uPi_{X^n}$, if it exists, is defined as \citep[][App.~C.2]{lower.previsions.book} 
\[ \uPi_{X^n} f := \int_0^\infty \uPi_{X^n}(f \geq t) \, dt - \int_{-\infty}^0 \{ 1 - \uPi_{X^n}(f \geq t) \} \, dt, \]
where the integrals on the right-hand side are of the usual Riemann variety.  To keep this discussion here relatively simple, we will focus on non-negative functions $f$, so the second term on the right-hand side above can be dropped.  The remaining integrand can be simplified by using the basic properties of possibility measures: 
\begin{equation}
\label{eq:choquet}
\uPi_{X^n}f = \int_0^\infty \Bigl\{ \sup_{\theta: f(\theta) \geq t} \pi_{X^n}(\theta) \Bigr\} \, dt. 
\end{equation}
Equivalent expressions for the Choquet integral with respect to a possibility measure \citep[e.g.,][Prop.~7.14]{lower.previsions.book} have been used in other IM applications \citep[e.g.,][]{imdec}; the above expression is particularly convenient for the point that we are making here. Of course, if $f$ is the indicator function of $H$, then 
\[ \sup_{\theta: f(\theta) \geq t} \pi_{X^n}(\theta) = \begin{cases} \uPi_{X^n}(H) & \text{if $t \in [0,1]$} \\ 0 & \text{otherwise}, \end{cases} \]
so we recover $\uPi_{X^n}f = \uPi_{X^n}(H)$ from \eqref{eq:choquet} as expected.  

To our point, which we will only sketch here, it follows from the definition \eqref{eq:choquet} of the Choquet integral and the simple property \eqref{eq:sups} that 
\[ \bigl| \uPi_{X^n} f - \uGamma_{X^n}f \bigr| \leq \int_0^\infty \sup_{\theta: f(\theta) \geq t} \bigl| \pi_{X^n}(\theta) - \gamma_{X^n}(\theta) \bigr| \, dt. \]
Then the uniform convergence result \eqref{eq:bvm} established in Theorem~\ref{thm:bvm} assists in showing that the upper bound is vanishing and, hence, that the Choquet integrals of $f$ with respect to $\uPi_{X^n}$ and the Gaussian approximation $\uGamma_{X^n}$ merge asymptotically.  For example, if $f$ is a continuous function supported on a compact set $\K$, then $\sup f$ is finite, 
\[ \bigl| \uPi_{X^n} f - \uGamma_{X^n}f \bigr| \leq (\sup f)   \times \sup_{\theta \in \K} \bigl| \pi_{X^n}(\theta) - \gamma_{X^n}(\theta) \bigr|, \]
and the right-hand side vanishes in $\prob_\Theta$-probability as $n \to \infty$.  Of course, even stronger results along this line are possible, including the following generalization of Corollary~\ref{cor:consistency}.  

\begin{cor}
\label{cor:cmt}
Let $f$ be a continuous function on a domain that contains $\Theta$ in its interior.  Under the conditions of Theorem~\ref{thm:bvm}, $\uPi_{X^n} f \to f(\Theta)$ in $\prob_\Theta$-probability. 
\end{cor}


\subsection{Technical remarks}
\label{SS:remarks}

\begin{remark}[Bayesian--fiducial--IM connections]
\label{re:credal}
\citet{martin.isipta2023} established a connection between the Bayesian and fiducial solutions and the possibilistic IM solution for the important class of group transformation models \citep[e.g.,][]{eaton1989, schervish1995}, which includes location-scale models, etc.  Specifically, for such models, the default-prior Bayes and fiducial solutions agree, and this common Bayes/fiducial solution was shown to be the inner probabilistic approximation of the possibilistic IM; in other words, the Bayes/fiducial distribution is the ``most diffuse'' member of the possibilistic IM's credal set.  While this Bayes/fiducial--IM connection only holds {\em exactly} for the aforementioned special class of models, Theorem~\ref{thm:bvm} implies that it holds {\em asymptotically} for all regular models.  We like to interpret this through the lens of {\em limit experiments}; see, e.g., \citet[][Ch.~10]{lecam.1986.book} and \citet{vaart1998, vaart.lecam.2002}.  These results say, roughly, that inference in a regular model asymptotically corresponds to inference in a Gaussian location model.  So, if the Bayes/fiducial--to--possibilistic--IM connection holds in Gaussian location models, as shown in \citet{martin.isipta2023}, then one would expect it also to hold asymptotically for all sufficiently regular models.  Theorem~\ref{thm:bvm} confirms this expectation.  

This credal set connection also sheds light on the importance and relevance of the IM's possibilistic structure.  First, recall that statistical inference involves ruling out hypotheses that are implausible based on the data, and validity protects against erroneously assigning small plausibility to true hypotheses (or large belief to false hypotheses).  Then {\em efficiency} in the statistical sense that is relevant to us here concerns how implausible those hypotheses in the tails, away from the true $\Theta$, are deemed to be.  Thanks to the characterization result quoted in \eqref{eq:credal}, the possibilistic IM's credal set---and, hence, the possibilistic IM itself---is completely determined by the tails of the contour function $\pi_{x^n}$.  So, to achieve asymptotic efficiency, it is sufficient that the tails of the IM's contour match the approximately-Gaussian tails of the asymptotically efficient Bayesian posterior.  From this perspective, we find that the limited expressiveness of the possibilistic IM is a virtue: all we have to do is cover the tails efficiently for large $n$, and let consonance take care of validity for each fixed $n$.  With a more expressive quantification of uncertainty, a greater burden would be imposed on the practitioner to balance validity and efficiency.
\end{remark}

\begin{remark}[Conditions of Theorem~\ref{thm:bvm}]
There are some differences between the more-classical Bayesian and fiducial Bernstein--von Mises theorems compared to that above for possibilistic IMs.  The reason is that the Bayesian and fiducial solutions do not directly rely on the maximum likelihood estimator---as the possibilistic IM does in the formation of the relative likelihood---so the corresponding Bernstein--von Mises theorems do not need to make assumptions such as consistency of the maximum likelihood estimator.  One can rectify this difference by replacing the relative likelihood above with something else, e.g., a different kind of normalized likelihood, 
\begin{equation}
\label{eq:normalized.lik}
(x^n, \theta) \mapsto \frac{L_{x^n}(\theta)}{\int L_{x^n}(\vartheta) \, W(d\theta)}, 
\end{equation}
where $W$ is a probability distribution supported on $\TT$.  The large-sample properties of this normalized likelihood \eqref{eq:normalized.lik} would be similar to those of the relative likelihood used here, so we conjecture that the use of this in place of our relative likelihood would produce a different possibilistic IM with the same large-sample properties as in Theorem~\ref{thm:bvm}, but under weaker conditions.  That being said, a further conjecture is that the possibilistic IM based on the normalized likelihood \eqref{eq:normalized.lik} would tend to be less efficient than that based on the relative likelihood in finite samples.  Together, our conjectures suggest that there is no free lunch: one can relax the conditions of Theorem~\ref{thm:bvm} by replacing the relative likelihood with something that does not directly depend on the maximum likelihood estimator, but the price for this relaxation would be a general loss of efficiency in finite samples.  We leave investigation into these (and other) conjectures as a topic for future work.  
\end{remark}

\begin{remark}[Connection to Gaussian random fuzzy numbers]
Interesting connections can be drawn between the possibilistic IM solution, via Theorem~\ref{thm:bvm}, and Thierry Den{\oe}ux's recent framework for uncertain reasoning based on Gaussian random fuzzy numbers.  Starting with the relative likelihood, apply both a standard Taylor approximation and the asymptotic equivalence of the observed and expected Fisher information to get  
\[ R(X^n, \theta) = \exp\{ \ell_{X^n}(\theta) - \ell_{X^n}(\hat\theta_{X^n}) \} \approx \exp\{-\tfrac12 (\theta - \hat\theta_{X^n})^\top (n I_\Theta) (\theta - \hat\theta_{X^n})\}. \]
The right-hand side is the membership function of a Gaussian fuzzy number with mean vector $\hat\theta_{X^n}$ and precision (inverse covariance) matrix $n I_\Theta$.  Since the mean $\hat\theta_{X^n}$ of this Gaussian fuzzy number has (asymptotic) distribution $\nm_D(\Theta, (n I_\Theta)^{-1})$, this has the makings of a Gaussian random fuzzy number as developed in \citet{denoeux.fuzzy.2022, denoeux.fuzzy.2023}.  That is, a belief and plausibility function on $\TT$ could be constructed by averaging the $\hat\theta_{X^n}$-dependent belief and plausibility functions derived from the above Gaussian membership function with respect to the large-sample Gaussian approximation of the sampling distribution of $\hat\theta_{X^n}$.  But our goal and approach are different. We are interested in quantifying uncertainty about $\Theta$, given $X^n$, so our result must depend on $X^n$, i.e., we do not average over $X^n$.  Rather than averaging over $\hat\theta_{X^n}$, we apply a probability-to-possibility transform to the right-hand side of the above display, for fixed $\hat\theta_{X^n}$, resulting in the right-hand side of \eqref{eq:cr.bound.easy}. It is the use of the probability-to-possibility transform throughout the construction that gives the possibilistic IM solution its validity for all $n$, so it should come as no surprise that this transformation continues to play an active role even as $n \to \infty$. 
\end{remark}


\section{Cases involving nuisance parameters}
\label{S:nuisance}

\subsection{Setup}

Recall that we have formulated a statistical model indexed by a parameter $\theta$, taking values in $\TT$, and the goal is quantifying uncertainty about the unknown true value $\Theta$.  It is often the case, however, that a particular feature $\Phi = f(\Theta)$ of $\Theta$ is of primary practical interest.  For example, one might be solely interested in the mean survival time for patients; so, if the model is gamma, with shape and scale parameters, then the quantity of interest is $\text{mean} = \text{shape} \times \text{scale}$, the product.  At least intuitively, if we can quantify uncertainty about $\Theta$, then surely we can derive our quantification of uncertainty about $\Phi$ from this.  In all practical applications, however, the quantity of interest $\Phi$ will be of lower dimension than $\Theta$, so there is reason to expect that efficiency (statistical and/or computational) can be gained by directly targeting lower dimensional $\Phi$ rather than proceeding indirectly by going through the higher dimensional $\Theta$. 

Proceeding more formally, we will re-express the model parameter $\theta$ as a pair $(\phi, \lambda)$, taking values in $\FF \times \LL$, where $\phi$ is the {\em interest parameter} and $\lambda$ is the {\em nuisance parameter}.  Then the true $\Theta$ corresponds to $(\Phi, \Lambda)$.  Note that $\theta$ and $(\phi, \lambda)$ are in one-to-one correspondence, but $\theta$ and $\phi$ are not.  This is because $\phi$ is of lower dimension than $\theta$, and it is precisely this dimension difference that we want to take advantage of.  The question is how to eliminate the nuisance parameter so that we can make efficient marginal inference about $\Phi$.  A few different nuisance-parameter-elimination strategies are available and can be incorporated into the possibilistic IM construction; these are reviewed next.

\subsection{Eliminating nuisance parameters}

\subsubsection{Extension}

Perhaps the most natural approach to eliminate nuisance parameters is via the possibility-theoretic {\em extension principle} \citep[e.g.,][]{zadeh1975a}.  In particular, if $\pi_{x^n}(\theta)$ is the contour function corresponding to the possibilistic IM for the full model parameter $\Theta$, as discussed and analyzed in detail above, then the extension principle ({\sc ex}) defines corresponding marginal possibilistic IM for $\Phi$ based on the contour 
\[ \pi_{x^n}^\text{\sc ex}(\phi) = \sup_{\theta: f(\theta) = \phi} \pi_{x^n}(\theta), \quad \phi \in \FF. \]
This is just the possibilistic analog of integrating a joint probability density function to get the corresponding marginal; here, of course, the appropriate calculus is optimization rather than integration.  This is exactly what was done in Example~\ref{ex:logistic} above for marginal inference on the medial lethal dose in logistic regression. 
 Since $\theta$ and $(\phi,\lambda)$ are in one-to-one correspondence, we can  get ``$\pi_{x^n}(\phi,\lambda)$'' directly from $\pi_{x^n}(\theta)$ and then equivalently write the marginal contour as 
\[ \pi_{x^n}^\text{\sc ex}(\phi) = \sup_{\lambda \in \LL} \pi_{x^n}(\phi,\lambda), \quad \phi \in \FF. \]
Based on the discussion above, the reader might guess that this extension-based approach is not the most efficient---it is an indirect attack on the quantity of interest $\Phi$ through the full parameter $\Theta$.  We will show below that this intuition is correct.

\subsubsection{Profiling}

Recall that the relative likelihood for $\Theta$ is 
\[ R(x^n, \theta) = \frac{L_{x^n}(\theta)}{\sup_\vartheta L_{x^n}(\vartheta)}, \quad \theta \in \TT. \]
The profile relative likelihood for $\Phi = f(\Theta)$ is given by 
\[ R^\text{\sc pr}(x^n,\phi) = \sup_{\theta: f(\theta) = \phi} R(x^n, \theta) = \frac{\sup_{\vartheta: f(\vartheta)=\phi} L_{x^n}(\vartheta)}{\sup_\vartheta L_{x^n}(\vartheta)}, \quad \phi \in \FF. \]
The numerator of $R^\text{\sc pr}(x^n,\phi)$ is called the profile likelihood function for $\Phi$, hence the name {\em relative} profile likelihood for $R^\text{\sc pr}(x^n,\phi)$.  Using the reparametrization $\theta \to (\phi,\lambda)$, the relative profile likelihood can be written as 
\[ R^\text{\sc pr}(x^n, \phi) = \frac{\sup_\lambda L_{x^n}(\phi, \lambda)}{\sup_{\varphi, \lambda} L_{x^n}(\varphi, \lambda)} = \frac{L_{x^n}(\phi, \hat\lambda_{x^n}(\phi))}{L_{x^n}(\hat\phi_{x^n}, \hat\lambda_{x^n})}, \]
where $(\hat\phi_{x^n}, \hat\lambda_{x^n})$ corresponds to the maximum likelihood estimator $\hat\theta_{x^n}$ and $\hat\lambda_{x^n}(\phi)$ is the conditional maximum likelihood estimator of the nuisance parameter when the interest parameter is fixed at $\phi$. 

The key take-away message in, e.g., \citet{murphy.vaart.2000}, is that, at least asymptotically, and under regularity conditions satisfied by most of the models commonly used in practice, the (relative) profile likelihood can be treated as if it were a genuine (relative) likelihood: among other things, $-2\log R^\text{\sc pr}(X^n,\Phi)$ is approximately chi-square.  Following \citet{martin.partial3}, here we propose to construct a marginal possibilistic IM for $\Phi$ by using the relative profile likelihood exactly how we used the relative likelihood in Section~\ref{SS:im}, i.e., by applying the probability-to-possibility transform: 
\begin{align}
\label{eq:mpl.pr0}
\pi_{x^n}^\text{\sc pr}(\phi) & = \sup_{\theta: f(\theta) = \phi} \prob_\theta\bigl\{ R^\text{\sc pr}(X^n, \phi) \leq R^\text{\sc pr}(x^n, \phi) \bigr\} \nonumber \\
& = \sup_{\lambda \in \LL} \prob_{\phi,\lambda}\bigl\{ R^\text{\sc pr}(X^n, \phi) \leq R^\text{\sc pr}(x^n, \phi) \bigr\}, \quad \phi \in \FF, 
\end{align}
where $\prob_{\phi,\lambda}$ is the model $\prob_\theta$ but now expressed in terms of $(\phi,\lambda)$.  The latter expression adds a bit of transparency that can be helpful in certain cases.  For example, if it happened that $R(X^n,\phi)$ was a pivot with respect to $\prob_{\phi,\lambda}$, then the outer supremum over $\lambda$ in the above display can be dropped, reducing the computational complexity, etc.  The result from Murphy and van der Vaart mentioned above says specifically that the relative profile likelihood is an asymptotic pivot, so the $\prob_{\phi,\lambda}$-probability on the right-hand side of \eqref{eq:mpl.pr0} is largely insensitive to $\lambda$.  This means that, for moderate to large samples, we can effectively replace the supremum over $\lambda$ by fixing $\lambda$ at any convenient value, which creates a practical opportunity to enjoy the aforementioned reduction in computational complexity. 


\subsubsection{Quick illustration}

Let $X^n$ denote an iid sample from a normal distribution with unknown mean $\Phi \in \FF = \RR$ and unknown variance $\Lambda \in \LL = \RR^+$.  As the notation suggests, while the entirety of $\Theta = (\Phi, \Lambda)$ is unknown, the goal is inference on the mean $\Phi$.  Data $x^n$ consists of $n=15$ observations with maximum likelihood estimators $\hat\phi_{x^n} = 3$ and $\hat\lambda_{x^n} = 1.5$.  There is no closed-form expression for the joint contour for $(\Phi,\Lambda)$ in this normal example, but this can easily be evaluated on an arbitrary grid of $(\phi, \lambda)$ values via Monte Carlo, since the relative likelihood function is a pivot.  Figure~\ref{fig:normal}(a) shows this joint contour; as expected, its peak is at the maximum likelihood estimator $(3,1.5)$.  The extension-based marginal IM possibility, shown in Figure~\ref{fig:normal}(b), is also easy to evaluate numerically from the joint contour: just maximize over the grid of $\lambda$ values at each $\phi$ value on its grid.  For comparison, the profile-based marginal IM contour---which actually has a closed-form expression \citep[see][Example~3]{martin.partial3}---is shown in red in Figure~\ref{fig:normal}(b).  Notice how the profile-based contour is narrower than the extension-based contour; and since both are valid by construction, it follows that the former is more efficient than the latter.  It is precisely this observation, which is not unique to this normal example, that suggested the conjecture that profile- is more efficient than extension-based marginalization.  

\begin{figure}[t]
\begin{center}
\subfigure[Joint contour for $(\Phi,\Lambda)$]{\scalebox{0.6}{\includegraphics{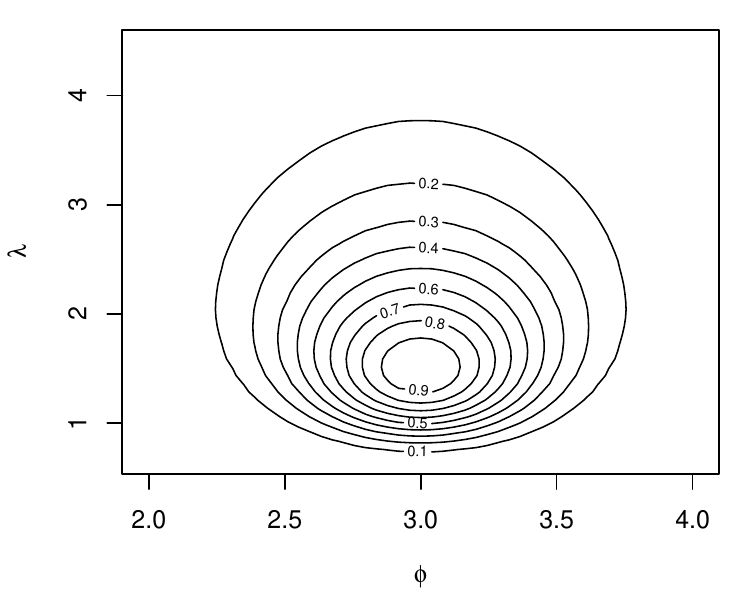}}}
\subfigure[Marginal contour for $\Phi$]{\scalebox{0.6}{\includegraphics{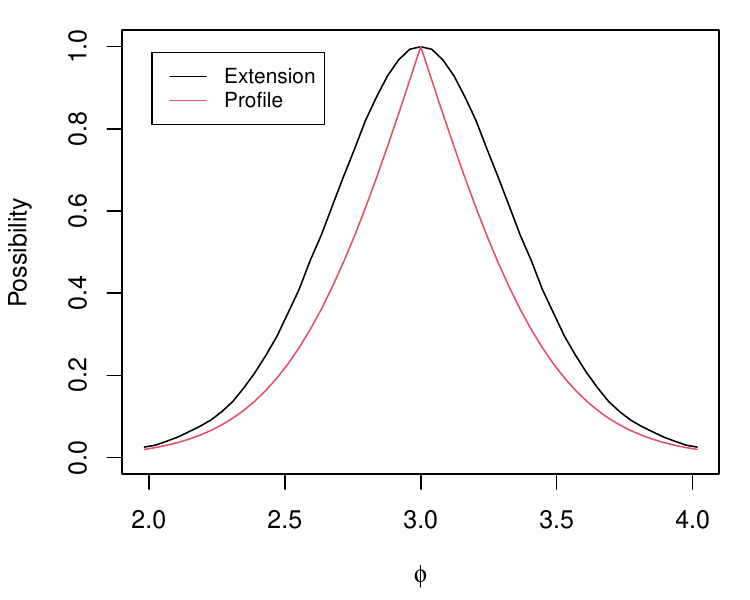}}}
\end{center}
\caption{Data consists of $n=15$ observations from a $\nm(\Phi,\Lambda)$ model, with observed maximum likelihood estimators $\hat\phi_{x^n}=3$ and $\hat\lambda_{x^n}=1.5$. Panel~(a) shows the joint IM contour evaluated based on Monte Carlo, while Panel~(b) shows the extension- and profiling-based marginal IM contour for $\Phi$.}
\label{fig:normal}
\end{figure}


\subsubsection{Other strategies}

We will not go into details here, but it is at least worth mentioning that the two general procedures above are not the only ways to carry out marginalization.  One strategy that works on some occasions is {\em conditioning}, i.e., by replacing the relative likelihood with a relative conditional likelihood.  The point is that, in certain cases, by conditioning on a suitable statistic, the dependence on the nuisance parameter drops out, leaving only a function involving data and interest parameter; \citet{basu1977} refers to the statistic that is being conditioned on as {\em P-sufficient}, where ``P'' is for ``partial''---conditioning on the P-sufficient statistic only eliminates a part of the full parameter.  In other cases, one might ignore the likelihood altogether and build a possibilistic marginal IM from scratch by first identifying a solely interest-parameter-dependent function and then applying the probability-to-possibility transform.  Neither of these approaches are universal; whether they can be carried out successfully depends on the particular problem at hand, which is why we elect not to dive into the details here. 

In a different but related direction, the approach of \citet{immarg} offers a novel perspective on the elimination of nuisance parameters.  The authors' proposal is quite powerful in that it affords users the flexibility to do more-or-less whatever it takes---e.g., model-based strategies like marginalization and algebraic manipulations of data-generating equations---to eliminate the nuisance parameters and reduce the overall dimension of the problem so that efficient inference is within reach.  A disadvantage to their proposal is that, once the nuisance parameter elimination is done, inference requires specification of a suitable ``predictive random set.''  At present, guidelines on how to choose this random set to achieve both statistical and computational efficiency are lacking.

\subsection{Background}

Before presenting our marginal possibilistic Bernstein--von Mises theorem(s), we need some additional background on likelihood-based inference in the presence of nuisance parameters.  In particular, here we define versions of the score function and Fisher information with some tweaks that result from the interest--nuisance parameter decomposition.  

First, recall that the model density/mass function $p_\theta \equiv p_{\phi,\lambda}$ and the corresponding likelihood $L_{x^n}(\theta) \equiv \ell_{x^n}(\phi,\lambda)$ and log-likelihood $\ell_{x^n}(\phi,\lambda)$ are now being expressed in terms of the new parametrization $\theta \to (\phi, \lambda)$.  While these different forms are equivalent, it is important to be clear that $\phi$ and $\lambda$ need not be direct components of $\theta$, so differentiation with respect to $\phi$ and/or $\lambda$ implicitly involves a chain rule calculation; the point is that the details require more than plugging the corresponding $(\phi,\lambda)$ in for $\theta$ in all the expressions above. Write $\dot\ell_{x^n}(\phi,\lambda)$ to denote the derivative of $\ell_{x^n}(\phi,\lambda)$ with respect to $(\phi,\lambda)$; this is the score function for the full parameter $(\Phi,\Lambda)$.  Similarly, define the Fisher information matrix for the full parameter $(\Phi, \Lambda)$ as 
\[ (\phi,\lambda) \mapsto \begin{pmatrix} I_{11}(\phi,\lambda) & I_{12}(\phi,\lambda) \\ I_{21}(\phi,\lambda) & I_{22}(\phi,\lambda) \end{pmatrix}, \]
which, again, can be obtained by combining the original Fisher information matrix $I_\theta=I_{\phi,\lambda}$ with the transformation $\theta \to (\phi,\lambda)$ and the chain rule \citep[e.g.,][Eq.~6.16]{lehmann.casella.1998}.  In the above expression, the matrices $I_{11}(\phi,\lambda)$ and $I_{22}(\phi,\lambda)$ are square, of dimension $D_\phi$ and $D_\lambda$, respectively.  Moreover, if the value of $\Lambda$ happened to be known, then the Fisher information in a sample of size 1 would be given by the matrix $I_{11}(\Phi,\Lambda)$.

An additional wrinkle is that, for the common goal of inference on $\Phi$, there is a potentially dramatic difference between the case where $\Lambda$ is known compared to $\Lambda$ is unknown.  In particular, the latter case typically---and as expected---involves a loss of information compared to the former.  The way this is formalized is through the notion of {\em least-favorable submodels}.  The basic idea is as follows: the most efficient inference about $\Phi$ one can get when $\Lambda$ is unknown cannot be more efficient than---and, in fact, is exactly as efficient as---that corresponding to the worst or most difficult known-$\Lambda$ case.  A bit more specifically, one first constructs a submodel $\{\phi, \lambda(\phi)\}$ indexed solely by $\phi$, and then seeks to choose the mapping $\lambda(\cdot)$ such that the information associated with the corresponding submodel is minimal.  Even more specifically, this least-favorable submodel at the true $(\Phi, \Lambda)$ corresponds to the choice of mapping 
\[ \lambda(\phi) = \Lambda + \{ I_{22}^{-1}(\Phi,\Lambda) \, I_{21}(\Phi, \Lambda)\}^\top (\Phi - \phi), \quad \phi \in \FF. \]
With the same (innocent) abuse of notation as in the previous sections, write $\ell_{\phi,\lambda}(x) = \log p_{\phi,\lambda}(x)$.  Fusing the two parameters via the least-favorable submodel gives $\ell_\phi := \ell_{\phi, \lambda(\phi)}$.  Then the derivative $\tilde\ell_\phi$ of $\ell_\phi$ with respect to $\phi$ is given by 
\[ \tilde\ell_\phi = \tfrac{\partial}{\partial\phi} \ell_{\phi,\lambda} \bigr|_{\lambda=\lambda(\phi)} - \{I_{22}^{-1}(\Phi,\Lambda) \, I_{21}(\Phi, \Lambda)\}^\top \, \tfrac{\partial}{\partial\lambda} \ell_{\phi,\lambda} \bigr|_{\lambda=\lambda(\phi)}, \]
which is just a linear combination of the gradients of $(\phi, \lambda) \mapsto \ell_{\phi,\lambda}$ with respect to $\phi$ and $\lambda$, respectively.  At the true $\Phi$, the vector $\tilde\ell_\Phi = \tilde\ell_\Phi(X)$ is often called the {\em efficient score function} for the sample $X$, and it is easy to check that the expected value of $\tilde\ell_\Phi(X_1)$ is 0 when $X_1 \sim \prob_{\Phi,\Lambda}$, just like in the full model with the full parameter $\Theta$.  Moreover, the covariance matrix of $\tilde\ell_\Phi$, at the true $(\Phi, \Lambda)$, is easily shown to be 
\begin{equation}
\label{eq:efficient.fisher}
\tilde I_{\Phi,\Lambda} = I_{11}(\Phi,\Lambda) - I_{12}(\Phi,\Lambda) \, I_{22}^{-1}(\Phi,\Lambda) \, I_{21}(\Phi,\Lambda). 
\end{equation}
This is often called the {\em efficient Fisher information matrix}.  Notice that this matrix is generally smaller (relative to Loewner order) than the matrix $I_{11}(\Phi,\Lambda)$; this difference in the two information measures is the penalty resulting from $\Lambda$ being unknown rather than known, as the latter information matrix unrealistically assumes.  Note that the two matrices agree, i.e., $\tilde I_{\Phi,\Lambda} = I_{11}(\Phi,\Lambda)$, when $I_{12}(\Phi,\Lambda) = 0$; this corresponds to what is commonly referred to as {\em parameter orthogonality}, which we discuss this further following the statement (and proof sketch) of Theorem~\ref{thm:bvm.pr} below. 

As before, define a full-data, scaled version of the efficient score function as 
\begin{equation}
\label{eq:delta.marg}
\widetilde\Delta_{\Phi,\Lambda}(X^n) = n^{-1/2} \tilde I_{\Phi,\Lambda}^{-1} \sum_{i=1}^n \tilde\ell_\Phi(X_i),
\end{equation}
which converges in distribution to $\nm_{D_\phi}(0, I_{\Phi,\Lambda}^{-1})$ under $\prob_{\Phi,\Lambda}$, by the central limit theorem.  Under certain conditions, a centered and scaled version of the profile maximum likelihood estimator $\hat\phi_{X^n}$ of $\Phi$ is asymptotically equivalent to $\widetilde\Delta_{\Phi,\Lambda}(X^n)$, i.e., 
\begin{equation}
\label{eq:pmle}
n^{1/2}( \hat\phi_{X^n} - \Phi) = \widetilde\Delta_{\Phi,\Lambda}(X^n) + o_{\prob_{\Phi,\Lambda}}(1), \quad n \to \infty, 
\end{equation}
which, of course, implies asymptotic normality of $\hat\phi_{X^n}$, with asymptotic variance $\tilde I_{\Phi,\Lambda}^{-1}$, the inverse of the efficient Fisher information; see, e.g., Equation~(5) in \citet{murphy.vaart.2000}.  This is almost exactly as expected: asymptotic normality of $\hat\phi_{X^n}$ would typically follow from asymptotic normality of $\hat\theta_{X^n}$, together with the delta theorem.  We say ``almost'' because textbook-style application of the delta theorem would suggest the asymptotic variance is $I_{11}^{-1}(\Phi,\Lambda)$, which is typically too small compared to $\tilde I_{\Phi,\Lambda}^{-1}$.

\subsection{Marginal possibilistic Bernstein--von Mises}

As discussed above, our focus is on the case where $\Theta$ is finite-dimensional.  This, of course, implies that both the interest and nuisance parameters, i.e., $\Phi$ and $\Lambda$, are finite-dimensional too.  In this case, we can continue to work under the regularity conditions imposed by Theorem~\ref{thm:bvm}.  These regularity conditions can, however, become problematic in so-called {\em semiparametric problems} where $\Phi$ is finite-dimensional but $\Lambda$ is infinite-dimensional.  In such cases, the results quoted in the previous subsection (and even those presented below) can be established under weaker conditions that can be met in semiparametric problems \citep[e.g.,][]{murphy.vaart.2000}, but we will not consider this here.  There are practical challenges in formulating IMs for infinite-dimensional unknowns, so we save this as a topic for future investigation.  

To set the scene for our main results of this section, first, for simplicity, we will focus solely on the interesting local cases, i.e., along sequences of interest parameter values that are collapsing to $\Phi$ at a $n^{-1/2}$-rate.  More specifically, we consider only $\phi$ values of the form $\phi_n^z = \Phi + n^{-1/2} \, z$, indexed by bounded $z \in \RR^{D_\phi}$.  Second, recall that the profiling-based marginal possibility contour $\pi_{X^n}^\text{\sc pr}$, in \eqref{eq:mpl.pr0}, involved a supremum over all values of the nuisance parameter $\lambda \in \LL$.  This is needed solely because we do not know the value of the true $\Lambda$, and then only way to be {\em absolutely sure} that validity holds at $\Lambda$ is to consider strictly the ``worst case,'' hence largest over all $\lambda$ in $\LL$.  Here, however, we will restrict the supremum to a compact subset $\LL_0$ of $\LL$.  That is, we define the profiling-based marginal possibility contour for $\Phi$ as 
\begin{align}
\pi_{x^n}^\text{\sc pr}(\phi) & = \sup_{\lambda \in \LL_0} \prob_{\phi,\lambda} \{ R^\text{\sc pr}(X^n, \phi) \leq R^\text{\sc pr}(x^n, \phi) \} \notag \\
& = \sup_{\lambda \in \LL_0} G_n^{\phi,\lambda}\{ R^\text{\sc pr}(x^n, \phi) \}, \label{eq:mpl.pr}
\end{align}
where $G_n^{\phi,\lambda}$ denotes the distribution function of $R^\text{\sc pr}(X^n,\phi)$ when $X^n$ are iid $\prob_{\phi,\lambda}$.  This compactness restriction can be justified by an appeal to consistency: if $\hat\lambda_{X^n} = \hat\lambda_{X^n}(\hat\phi_{X^n})$ is, with high probability, close to $\Lambda$, then we can choose $\LL_0$ to be a large compact set that contains $\hat\lambda_{X^n}$ and be {\em pretty sure} that this covers $\Lambda$.  Third, we define the limiting Gaussian approximation under profiling, and under the local parametrization, as 
\begin{equation}
\label{eq:mpl.pr.lim}
\gamma_{X^n}^\text{\sc pr}(\phi_n^z) = G\bigl( \exp\bigl[-\tfrac12\{ z - \widetilde\Delta_{\Phi,\Lambda}(X^n)\}^\top \tilde I_{\Phi,\Lambda} \, \{ z - \widetilde\Delta_{\Phi,\Lambda}(X^n)\} \bigr] \bigr), 
\end{equation}
where $G$ denotes the distribution function of $\exp\{-\frac12 \chisq(D_\phi)\}$.  The following theorem establishes that $\pi_{X^n}^\text{\sc pr}$ in \eqref{eq:mpl.pr} and $\gamma_{X^n}^\text{\sc pr}$ are locally merging in a strong/uniform sense.  In words, the profiling-based marginal IM contour will locally and asymptotically resemble a Gaussian possibility contour with mean vector $\Phi + n^{-1/2} \widetilde\Delta_{\Phi,\Lambda}(X^n)$ and covariance matrix $(n \tilde I_{\Phi,\Lambda})^{-1}$; and the aforementioned mean vector is, according to \eqref{eq:pmle}, equivalent to the profile maximum likelihood estimator $\hat\phi_{X^n}$. 

\begin{thm}[Profiling-based marginalization]
\label{thm:bvm.pr}
Let $(\Phi, \Lambda)$ reside in the interior of the parameter space $\FF \times \LL$, and let $\LL_0$ denote a compact subset of $\LL$ that contains $\Lambda$.  Under the conditions of Theorem~\ref{thm:bvm}, $\pi_{X^n}^\text{\sc pr}$ in \eqref{eq:mpl.pr} satisfies, for $\phi_n^z = \Phi + n^{-1/2} \, z$, 
\[ \sup_{z \in K} \bigl| \pi_{X^n}^\text{\sc pr}(\phi_n^z) - \gamma_{X^n}^\text{\sc pr}(\phi_n^z) \bigr| \to 0 \quad \text{in $\prob_{\Phi,\Lambda}$-probability}, \]
for any compact $K \subset \RR^{D_\phi}$.  
\end{thm}

\begin{proof}
A sketch of the proof is given in Appendix~\ref{SS:bvn.pr.proof}.
\end{proof}

An interesting observation is that, in the case where $I_{12}=0$, the profile-based marginal IM has an asymptotic Gaussian approximation that agrees with that achievable when the nuisance parameter $\Lambda$ is {\em known}.  This agreement is what \citet{bickel1998} refer to as {\em adaptation} in the context of parameter estimation.  Of course, it is not possible to make better or more efficient inference about $\Phi$ when $\Lambda$ is unknown compared to when $\Lambda$ is known, so this known-$\Lambda$ scenario is a kind of gold-standard.  So, it is remarkable that the profile-based marginal IM is able to achieve adaptive efficiency in these cases.  The condition ``$I_{12}=0$'' is often referred to as {\em parameter orthogonality} as in \citet{CoxReid:1987}, \citet{royall.book}, and elsewhere.  The most familiar example of orthogonal parameters is Gaussian where interest is in the mean (or in the variance).  Usually, the feature $\phi = f(\theta)$ is set by the context of the problem, but the choice of parametrization $\theta \to (\phi,\lambda)$ is at the discretion of the data analyst; in that case, one would be free to choose $\lambda$ such that $\phi$ and $\lambda$ are orthogonal, if possible.  \citet[][Sec.~3]{CoxReid:1987} show a number of such examples, but there are cases where no orthogonal parametrization is available, hence adaptively efficient inference would be out of reach.  

But recall that the profiling-based marginalization is not the only strategy available.  How does the simpler and arguably more natural extension-based marginalization compare asymptotically to the profiling-based marginalization in Theorem~\ref{thm:bvm.pr}?  The next theorem gives the detailed answer, but the take-away message is that extension-based marginalization is less efficient than profiling-based marginalization.  

Just like Theorem~\ref{thm:bvm.pr} above, we have a bit of scene-setting to do before presenting the result.  First, we will again focus our attention on the most interesting local behavior, for $\phi$ of the form $\phi_n^z = \Phi + n^{-1/2} \, z$ for $z$ in a compact set.  Second, we will again restrict attention to a compact subset $\LL_0$ of the nuisance parameter space $\LL$, so that the extension-based marginal IM contour is given by 
\begin{equation}
\label{eq:mpl.ex}
\pi_{x^n}^\text{\sc ex}(\phi) = \sup_{\lambda \in \LL_0} \pi_{x^n}(\phi,\lambda), \quad \phi \in \FF, 
\end{equation}
where $\pi_{x^n}(\phi,\lambda)$ is the original IM possibility contour for $\Theta$ but now re-expressed in terms of the interest--nuisance parameter pair.  Again, this can be justified by a consistency argument.  Third, we define the limiting Gaussian approximation under extension, and under the local parametrization, as 
\begin{equation}
\label{eq:mpl.ex.lim}
\gamma_{X^n}^\text{\sc ex}(\phi_n^z) = G\bigl( \exp\bigl[-\tfrac12\{ z - \widetilde\Delta_{\Phi,\Lambda}(X^n)\}^\top \tilde I_{\Phi,\Lambda} \, \{ z - \widetilde\Delta_{\Phi,\Lambda}(X^n)\} \bigr] \bigr). 
\end{equation}
The above expression is identical to \eqref{eq:mpl.pr.lim}, but there is an implicit difference: here $G$ denotes the distribution function of $\exp\{-\frac12 \chisq(D)\}$, where the degrees of freedom in the aforementioned chi-square is $D := D_\phi + D_\lambda$, potentially much greater than $D_\phi$.  

\begin{thm}[Extension-based marginalization]
\label{thm:bvm.ex}
Let $(\Phi, \Lambda)$ reside in the interior of the parameter space $\FF \times \LL$, and let $\LL_0$ denote a compact subset of $\LL$ that contains $\Lambda$.  Under the conditions of Theorem~\ref{thm:bvm}, $\pi_{X^n}^\text{\sc ex}$ in \eqref{eq:mpl.ex} satisfies, for $\phi_n^z = \Phi + n^{-1/2} \, z$, 
\[ \sup_{z \in K} \bigl| \pi_{X^n}^\text{\sc ex}(\phi_n^z) - \gamma_{X^n}^\text{\sc ex}(\phi_n^z) \bigr| \to 0 \quad \text{in $\prob_{\Phi,\Lambda}$-probability}, \]
for any compact $K \subset \RR^{D_\phi}$. 
\end{thm}

\begin{proof}
See Appendix~\ref{SS:bvn.ex.proof}. 
\end{proof}

Our claim above was that the extension-based marginalization strategy is less efficient than the profiling based strategy; in fact, it can be substantially less efficient.  But the reason is not particularly deep: looking at the difference between the two Gaussian limiting contours, the only difference is in the degrees of freedom in the relevant chi-square distributions.  For the profiling-based strategy, the degrees of freedom is exactly the dimension of the interest parameter, $D_\phi$.  For the extension-based marginalization strategy, on the other hand, the degrees of freedom is the full parameter dimension $D = D_\phi + D_\lambda$.  The effect of this dimension mismatch is that the extension-based contour will not be as peaked at the maximum likelihood estimator as the profiling-based contour, making the former overall larger and, thereby, less efficient.  This is exactly what we saw in Figure~\ref{fig:normal} for the simple normal mean model. All that being said, the comparison between the two strategies is not exactly fair, so this should not be interpreted as a criticism.  The reason why this comparison is unfair is that the apparent inefficiency in the extension-based marginalization scheme is fully expected.  The point is that the profiling-based strategy requires the user to commit to a particular interest parameter, and put in the effort to derive the marginal IM contour.  The extension-based strategy, however, can immediately adapt to whatever interest parameter is relevant to the user, all that he/she needs is a suitable optimization routine.  



\subsection{Numerical illustrations}


\begin{ex}
\label{ex:reg.variance}
Consider a typical Gaussial linear regression model, where the response variables, which we will denote here as $X_i$'s, are independent with 
\[ (X_i \mid u_i) \sim \nm(u_i^\top \Lambda, \Phi), \quad i=1,\ldots,n, \]
where $u_i \in \RR^r$ is a fixed vector of covariates, $\Lambda \in \RR^r$ is an unknown vector of regression coefficients, and $\Phi > 0$ is an unknown variance.  In vector form, we have 
\[ (X \mid U) \sim \nm_n(U\Lambda, \Phi \, I_{n \times n}), \]
where $U$ is an $n \times r$ design matrix.  Here we treat the case where $r < n$ is fixed, and assume that $U$ has full rank $r$.  As the notation indicates, our interest is in the variance $\Phi$, hence the vector of regression coefficients $\Lambda$ is a nuisance parameter.  We proceed to construct a marginal possibilistic IM for $\Phi$ using profiling.  

The likelihood function for this standard Gaussian linear regression model is well-known and given by the formula 
\[ L_{x}(\phi, \lambda) = (2\pi \phi)^{-n/2} \exp\bigl\{ -\tfrac{1}{2\phi} \|x - U\lambda\|^2 \bigr\}. \]
Let $P = U(U^\top U)^{-1} U^\top$ denote the matrix that projects onto the space spanned by the columns of $U$, so that $Px = U\hat\lambda_x$, where $\hat\lambda_x$ is the usual least squares estimator.  Then 
\[ \|x - U\lambda\|^2 = \|x - Px + Px - U\lambda\|^2 = \|x - Px\|^2 + \|U\hat\lambda_x - U\lambda\|^2, \]
by the Pythagorean theorem.  This makes it clear that the likelihood $L_x(\phi,\lambda)$ factors into a product of a term that depends on $\lambda$ and a term that does not depend on $\lambda$.  Indeed, 
\[ \sup_\lambda L_x(\phi,\lambda) = (2\pi \phi)^{-n/2} \exp\bigl\{ -\tfrac{1}{2\phi} \|x - Px\|^2 \bigr\}, \]
and, therefore, the relative profile likelihood is 
\[ R^\text{\sc pr}(x, \phi) = \exp\bigl[ -\tfrac{n}{2}\{\hat\phi_x / \phi - 1 - \log(\hat\phi_x/\phi)\} \bigr], \]
where $\hat\phi_x = n^{-1} \|(I-P)x\|^2$ is the maximum likelihood estimator, the usual residual sum of squares scaled by $n^{-1}$.  The key observation is that, if $\Phi=\phi$, then $\|(I-P)X\|^2 / \phi$ is a pivot---it is distributed as $\chisq(n-r)$ irrespective of the values of $\Phi$ or $\Lambda$.  This makes it easy to obtain the profile-based marginal possibilistic IM contour for $\Phi$ via Monte Carlo.  For an illustration, consider the 1970 Boston housing prices data, originally due to \citet{harrison.rubinfeld.1978}, now readily available in the {\tt MASS} function in R. This data consists of information in $n=506$ census tracts, which includes the quantity of primary interest, namely, the median value of owner-occupied homes, and the values of 13 other relevant predictor variables.  A simple linear regression model was fit and the maximum likelihood estimator for the error variance is $\hat\phi_x = 21.94$.  Figure~\ref{fig:reg.variance} shows a plot of the corresponding marginal IM contour for $\Phi$ based on these data.  

For comparison, we can also easily get the corresponding large-sample approximation of the profile-based marginal IM contour.  In this case, the parameters $\Phi$ and $\Lambda$ are orthogonal (i.e., the off-diagonal blocks of the Fisher information matrix are all zeros), so the efficient Fisher information is $I_{11}(\Phi,\Lambda) = (2\Phi^2)^{-1}$; this means that our results are adaptive in the sense that marginal inference on $\Phi$, with unknown $\Lambda$, is as efficient as in the case where $\Lambda$ is known.  Figure~\ref{fig:reg.variance} also shows the asymptotically approximate marginal IM contour for $\Phi$, which closely agrees with the exact contour, at least in the left tail; in the right tail, there is some skewness in the exact contour that the necessarily-symmetric large-sample approximation cannot pick up. 
\end{ex}

\begin{figure}[t]
\begin{center}
\scalebox{0.7}{\includegraphics{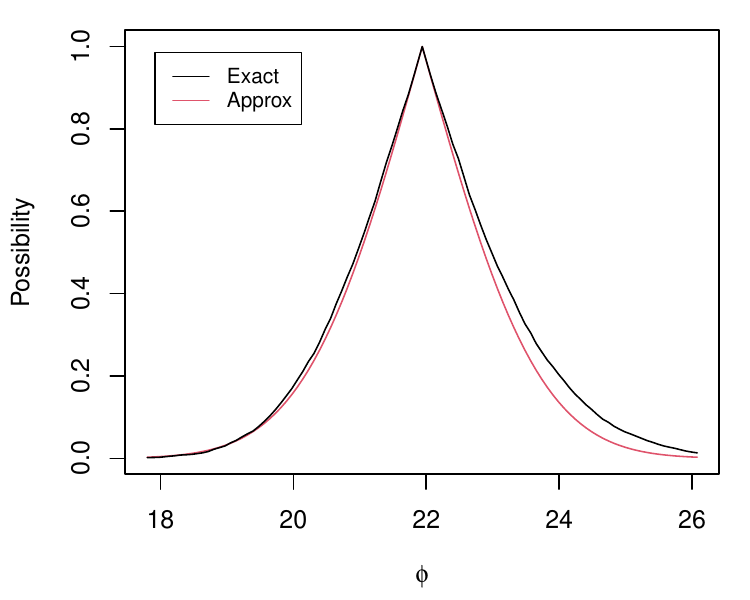}}
\end{center}
\caption{Plot of the exact and approximate profile-based marginal IM contour for the error variance in Gaussian linear regression as described in Example~\ref{ex:reg.variance}, based on a model fit to the Boston housing prices data.}
\label{fig:reg.variance}
\end{figure}

\begin{ex}
\label{ex:gamma.mean}
A surprisingly challenging problem is inference on the mean of a gamma distribution; see, e.g., \citet{shiue.bain.1990}, \citet{fraser.reid.wong.1997}, \citet{martin.partial3}, and the references therein.  Let $X^n=(X_1,\ldots,X_n)$ be an iid sample of size $n$ from a gamma distribution with unknown shape parameter $\Theta_1 > 0$ and unknown scale parameter $\Theta_2 > 0$; write $\Theta=(\Theta_1, \Theta_2)$ for the unknown shape--scale pair.  The quantity of interest is the mean, which, in the shape--scale parametrization, is given by the product $\Phi = \Theta_1 \Theta_2$.  Pairing this with $\Lambda = \Theta_2$, we can readily verify that $(\Phi,\Lambda)$ is in one-to-one correspondence with $\Theta$.  In terms of this mean--scale parametrization, the gamma density function is 
\[ p_{\phi,\lambda}(x) = \{ \lambda^{\phi/\lambda} \, \Gamma(\phi / \lambda)\}^{-1} \, x^{\phi/\lambda-1} \, e^{-x/\lambda}, \quad x > 0. \]
From here, the profiling-based marginal IM construction is at least conceptually straightforward, i.e., we can numerically evaluate the relative profile likelihood and use Monte Carlo to simulate its sampling distribution.  The details are described in Example~6 of \citet{martin.partial3} and we see no need to repeat those here.  But it is worth mentioning that, while the above construction is conceptually straightforward, the actual implementation is non-trivial and somewhat expensive, mostly because the sampling distribution of the relative profile likelihood depends on the nuisance parameter.  The computations can be carried out, however, and Figure~\ref{fig:gamma.mean} shows the corresponding marginal IM possibility contour based on a sample of size $n=30$ from a gamma distribution with $\Theta_1=9$ and $\theta_2=12$, so that the true mean is $\Phi=9 \times 12 = 108$.  The corresponding 95\% confidence interval is, roughly, $(104, 130)$, which happens to contain the true value 108 in this example.  To our knowledge, this marginal IM solution is the best among those other provably, exactly-valid methods available in the literature for inference on the gamma mean. 

For comparison, we can also get an asymptotic approximation.  For this we need the Fisher information, $I_{\phi,\lambda}$, in the mean--scale parametrization.  This can be derived directly from the expression for the gamma density function given above; or one can use the aforementioned chain-rule formula, which involves derivatives of $\theta$ with respect to $(\phi,\lambda)$, to get from the Fisher information in the shape--scale parametrization to that in the mean--scale parametrization.  Either way, the expression is quite messy, and offers no insight, so it is not worth writing down here; the numerical values and any specified $(\phi,\lambda)$ pair, however, are easy to get.  Specifically, we can plug in the maximum likelihood estimators to get $\tilde I_{\hat\phi_{X^n}, \hat\lambda_{X^n}}$ as in \eqref{eq:efficient.fisher}, which can be used to approximate $\tilde I_{\Phi,\Lambda}$ in \eqref{eq:mpl.pr.lim}.  Since the efficient score satisfies $\widetilde\Delta_{\Phi,\Lambda}(X^n)$ satisfies \eqref{eq:pmle}, we can approximate the profiling-based marginal IM contour by 
\[ \pi_{x^n}^\text{\sc pr}(\phi) \approx G\bigl( \exp\bigl\{ -\tfrac12 (\phi - \hat\phi_{x^n})^\top \tilde I_{\hat\phi_{x^n}, \hat\lambda_{x^n}} (\phi - \hat\phi_{x^n}) \bigr\} \bigr), \quad \phi \in \FF, \]
where $G$ is the distribution function of $\exp\{-\frac12 \chisq(1)\}$.  Figure~\ref{fig:gamma.mean} also shows this approximate contour for the same data generated as described above.  Notice that the exact contour has a bit of right skewness in the tail that the symmetric approximation is unable to accommodate.  Overall, however, the approximation is quite accurate, and would only improve with a larger sample size. 
\end{ex}

\begin{figure}[t]
\begin{center}
\scalebox{0.7}{\includegraphics{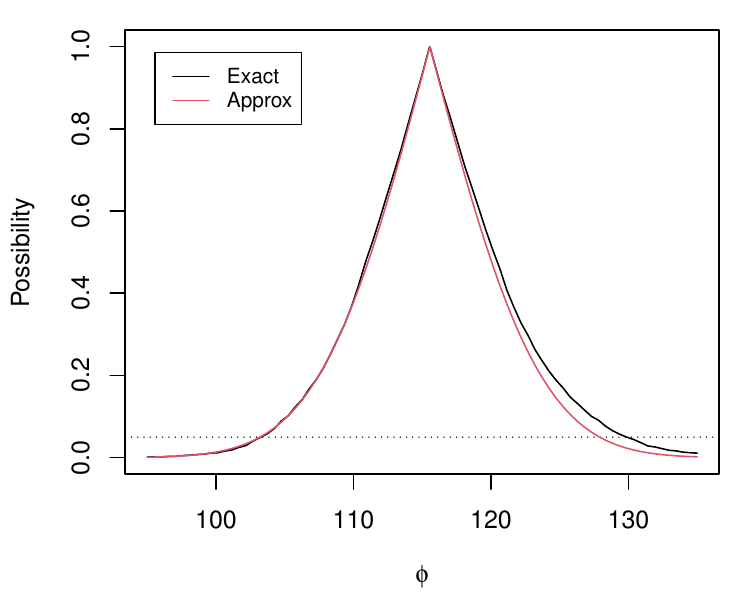}}
\end{center}
\caption{Plot of the exact and approximate profile-based marginal IM contour for the gamma mean as described in Example~\ref{ex:gamma.mean}, where $n=30$ and the true mean is 108.}
\label{fig:gamma.mean}
\end{figure}

\begin{ex}
\label{ex:logistic.again}
Return to the logistic regression problem in Example~\ref{ex:logistic} above.  Suppose here that we are interested in the {\em median lethal dose}, $\Phi=-\Theta_0/\Theta_1$, i.e., the exposure level at which the probability of death is 0.5.  The dotted lines in Figure~\ref{fig:logistic}(a) show that the maximum likelihood estimator is $\hat\phi_n = 0.244$.  That is where we expect any marginal IM contour for $\Phi$ to be centered.  For inference on $\Phi$, a marginal possibility contour can be obtained in (at least) two ways: first, we can apply the extension principle to the large-sample Gaussian approximation as displayed in Figure~\ref{fig:logistic}(b) and, second, we can get the large-sample approximation of the profile-based marginal IM contour as in Theorem~\ref{thm:bvm.pr}.  We say ``at least'' because, in principle, there is an exact marginal IM contour, but this is computationally out of reach at the moment.  Figure~\ref{fig:logistic.again} compares these two approximation marginal IMs for $\Phi$ with the same mice data set as above.  That the profile-based approximation gives a tighter contour is to be expected based on the theory presented above---though the solution here might be a bit too aggressive here given that the sample size ($n=120$) is not-huge and $\Phi$ is a highly non-linear function of $\Theta$.  A point we did not mention before but is worth mentioning here is that the extension principle applied to a Gaussian approximation need not resulting in a Gaussian-looking contour, and this is apparent from the asymmetry in the extension-based contour in Figure~\ref{fig:logistic.again}.  The point is that the extension-based solution has certain benefits even if it fails in comparison to the profile-based solution in terms of asymptotic efficiency.  
\end{ex}

\begin{figure}[t]
\begin{center}
\scalebox{0.7}{\includegraphics{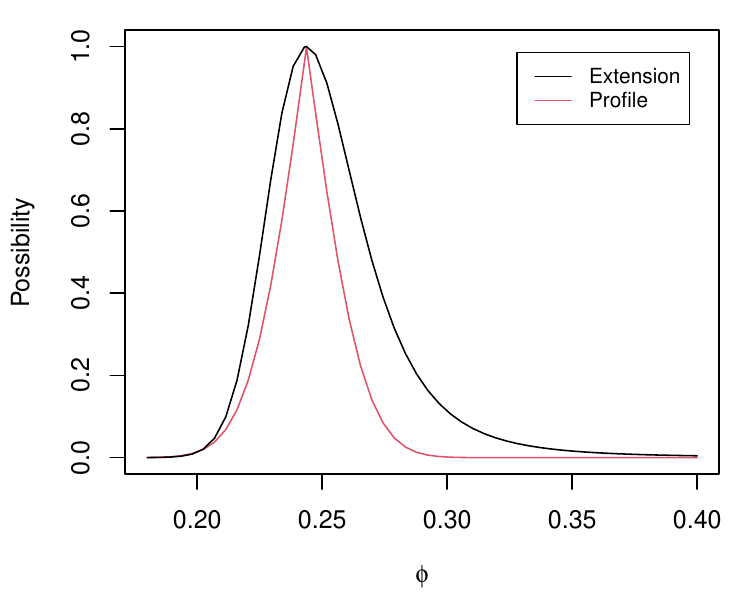}}
\end{center}
\caption{Results for logistic regression in Example~\ref{ex:logistic} and, more specifically, in Example~\ref{ex:logistic.again}.  Plot shows two versions of the marginal possibilistic IM contour for $\Phi = -\Theta_0/\Theta_1$, the median lethal dose.  One is based on applying the extension principle to the large-sample Gaussian contour in Figure~\ref{fig:logistic}(b), and the other is based on the large-sample approximation of the profile-based marginal IM contour.}
\label{fig:logistic.again}
\end{figure}





\section{Conclusion}
\label{S:discuss}

As explained in Section~\ref{S:intro}, it is natural for one to associate imprecision and exact, finite-sample validity with conservatism and inefficiency.  But it has been empirically shown that, to the contrary, the IM solutions---which are both imprecise and exactly valid---tend to be as efficient as those solutions which are only asymptotically valid.  The asymptotic efficiency of the IM solution was, until now, only a conjecture: here we established a new, possibility-theoretic version of the famous Bernstein--von Mises theorem which, among other things, establishes that the IM solutions are asymptotically efficient.  The key take-away message is that the exact validity offered by the IM solution, thanks to its inherent imprecision, requires no sacrifice in terms of asymptotic efficiency.  

On the practical side, the results here offer very simple approximations to the IM solutions.  These asymptotic versions, of course, are not exactly valid, but they are approximately so.  Then the question is if these simple approximations can be embellished on in some clever way such that exact validity can be achieved.  Intuitively, if we formally write down how to construct a provably valid IM in a vacuum, then apparently computation is a challenge; but if we have a pretty good idea what that IM looks like, then that could ease some of the computational burden.  Some preliminary results along these lines are presented in \citet{cella.martin.belief2024}, but further investigation is required.  

One possible extension of the results presented here is to the case where the uncertain $\Theta$ is defined as the minimizer of an expected loss.  This is a generalization of the setup in the present paper just like M-estimation or empirical risk minimization generalizes maximum likelihood estimation.  In this more general case, we replace the negative log-relative likelihood by the empirical regret (i.e., difference between the empirical risk and its minimum), and do the same validification step to construct a possibilistic IM contour \citep{cella.martin.imrisk}.  A similar possibilistic Bernstein--von Mises theorem is within reach for this empirical-risk-driven IM, and we will report on this elsewhere. 

Lastly, two adjacent directions to the possibilistic Bernstein--von Mises result developed here, for the possibilistic IM, are potential avenues for further research.  One direction is to investigate the asymptotic efficiency of the IM framework beyond the restriction to its construction as the probability-to-possibility transform of a relative likelihood, e.g., when a meaningful data generating equation and predictive random set are available.  Surely, these contexts with additional structure should also result in IM contour functions that are asymptotically efficient as in Theorem~\ref{thm:bvm}.  A related corollary based on such a result for the usual IM construction is a description of the asymptotic efficiency of conformal prediction sets, with reference to \cite{imconformal, imconformal.supervised}.  A second direction is the characterization of imprecise Bernstein--von Mises asymptotics beyond the context of possibility theory, to apply to lower-upper probability constructions more generally.  Such results might be of interest, for example, for reconciling belief-plausibility constructions from a subjectivist, Shaferian degrees of belief perspective.


\section*{Acknowledgments}

RM's research is supported by the U.S.~National Science Foundation, under grants SES--2051225 and DMS--2412628.

\appendix

\section{Technical details}

\subsection{Proof of Theorem~\ref{thm:bvm}}
\label{SS:bvn.proof}

To start, recall that the IM contour at a generic $\theta$ value can be expressed as 
\[ \pi_{X^n}(\theta) = G_n^\theta\{ R(X^n, \theta) \}, \quad \theta \in \TT, \]
where $R(X^n,\theta)$ is the relative likelihood and 
\[ G_n^\theta(r) = \prob_\theta\{ R(Y^n, \theta) \leq r \}, \quad r \in [0,1], \]
is the distribution function of the random variable $Y^n \mapsto R(Y^n,\theta)$ under the model with $Y^n$ iid $\prob_\theta$.  Note that $G_n^\theta$ itself does not depend on the data $X^n$, it is just a (possibly complicated) distribution function that depends on $n$ and $\theta$.  The dependence on $X^n$ in $\pi_{X^n}$ comes from the evaluation of $G_n^\theta$ at the relative likelihood $R(X^n,\theta)$.  Our choice to write ``$Y^n$'' in the notation above is to help make it clear that the function $G_n^\theta$ does not depend on the observable data $X^n$.  Then the famous result of \citet{wilks1938} establishes that, under the classical Cram\'er conditions, $-2\log R(Y^n, \theta)$ converges in distribution to a $\chisq(D)$ random variable under $\prob_\theta$, for any (interior point) $\theta \in \TT$.  So, if $G$ denotes the distribution function of the random variable $\exp\{-\frac12 \chisq(D)\}$, then it follows from Wilks and the continuous mapping theorem that $G_n^\theta(r) \to G(r)$ for each $r$.  By the fact that distribution functions are bounded and monotone, this pointwise convergence can be strengthened to uniform convergence:
\[ \|G_n^\theta - G\|_\infty := \sup_{r \in [0,1]} |G_n^\theta(r) - G(r)| \to 0. \]
We will argue below that this same result holds under the weaker regularity conditions assumed here; in fact, it turns out that an even stronger result holds, namely, that the convergence result in the above display holds locally uniformly in $\theta$.




The first key step in the following decomposition:
\begin{align}
\bigl| \pi_{X^n}(\theta) - \gamma_{X^n}(\theta) \bigr| & = \bigl| G_n^{\theta}\{ R(X^n, \theta)\} - \gamma_{X^n}(\theta) \bigr| \notag \\
& = \bigl| G_n^{\theta}\{ R(X^n, \theta) \} - G\{ R(X^n, \theta)\} \bigr| \notag + \bigr| G\{ R(X^n, \theta) \} - \gamma_{X^n}(\theta) \bigr| \notag \\
& \leq \bigl\| G_n^{\theta} - G \bigr\|_\infty + \bigl|  G\{ R(X^n, \theta) \} - \gamma_{X^n}(\theta) \bigr|. \label{eq:bound1}
\end{align}
This decomposition suggests a two-step proof: Lemmas~\ref{lem:step1} and \ref{lem:step2} below show that the first and second terms are suitably vanishing, which will prove the claim.

\begin{lem}
\label{lem:step1}
For any $\theta \in \TT$, let $G_n^\theta$ denote the distribution function of $R(Y^n,\theta)$ when $Y^n=(Y_1,\ldots,Y_n)$ consists of iid samples from $\prob_\theta$.  Under the stated conditions, 
\[ \| G_n^\theta - G \|_\infty \to 0 \quad \text{locally uniformly in $\theta$}, \]
where $G$ is the distribution function of $\exp\{-\frac12 \chisq(D) \}$.  
\end{lem}

\begin{proof}
For a generic data set $Y^n$ as introduced above, for a generic $\theta \in \TT$, and for a generic $u \in \RR^D$, write the log-likelihood process as 
\begin{equation}
\label{eq:lan}
\log \frac{L_{Y^n}(\theta + n^{-1/2} u)}{L_{Y^n}(\theta)} = u^\top I_\theta \Delta_\theta(Y^n) - \frac{u^\top I_\theta \, u}{2} + o_{\prob_\theta}(1), \quad n \to \infty, 
\end{equation}
where $\Delta_\theta(Y^n)$ is as in \eqref{eq:delta}.  The above property is called {\em local asymptotic normality} and is a consequence of regularity or, more specifically, of differentiability in quadratic mean \citep[e.g.,][Theorem~7.2]{vaart1998}.  The maximum likelihood estimator, $\hat\theta_{Y^n}$, is consistent by assumption, and $n^{-1/2}$-consistency follows from Corollary~5.53 in \citet{vaart1998}; that is, $n^{1/2}(\hat\theta_{Y^n}-\theta) = O_{\prob_\theta}(1)$.  Therefore, there exists a constant sequence $M_n$, with $M_n \to \infty$, such that $n^{1/2}\|\hat\theta_{Y^n}-\theta\| \leq M_n$ with $\prob_\theta$-probability converging to 1; technically, this $M_n$ can depend on the particular $\theta$, but there are choices of $M_n$ that meet the above requirements simultaneously for all $\theta$ in a given compact set, and so we take $M_n$ to be such a sequence.  Then, with probability converging to 1, the usual likelihood ratio statistic can be written as 
\[ -2 \log R(Y^n, \theta) = 2 \log \frac{L_{Y^n}(\hat\theta_{Y^n})}{L_{Y^n}(\theta)} = 2 \sup_{u: \|u\| \leq M_n} \log\frac{L_{Y^n}(\theta + n^{-1/2} u)}{L_{Y^n}(\theta)}. \]
In the proof of Theorem~7.12 in \citet{vaart1998}, the approximation in \eqref{eq:lan} is shown to hold uniformly over $u$ with $\|u\| \leq M_n$, so we get 
\[ \sup_{u: \|u\| \leq M_n} \Bigl| 2 \log\frac{L_{Y^n}(\theta + n^{-1/2} u)}{L_{Y^n}(\theta)} - \bigl\{ 2u^\top I_\theta \Delta_\theta(Y^n) - u^\top I_\theta \, u \bigr\} \Bigr| \to 0, \quad \text{in $\prob_\theta$-probability}. \]
By the bound in \eqref{eq:sups} and the representation of the likelihood ratio statistic above, the result in the above display implies 
\[\Bigl| -2\log R(Y^n, \theta) - \sup_{u: \|u\| \leq M_n} \bigl\{ 2u^\top I_\theta \Delta_\theta(Y^n) - u^\top I_\theta \, u \bigr\} \Bigr| \to 0, \quad \text{in $\prob_\theta$-probability}. \]
The inner supremum can be easily evaluated in closed-form, since the function is just a quadratic in $u$: 
\[ \sup_{u: \|u\| \leq M_n} \bigl\{ 2u^\top I_\theta \Delta_\theta(Y^n) - u^\top I_\theta \, u \bigr\} = \Delta_\theta(Y^n)^\top \, I_\theta \, \Delta_\theta(Y^n). \]
Since convergence in probability implies convergence in distribution, and we know from \eqref{eq:score.clt.unif} that $\Delta_\theta(Y^n) \to \nm_D(0, I_\theta^{-1})$ in distribution  under $\prob_\theta$, it follows that $\Delta_\theta(Y^n)^\top \, I_\theta \, \Delta_\theta(Y^n)$ and, hence, $-2\log R(Y^n, \theta)$, converges in distribution to $\chisq(D)$ under $\prob_\theta$.  Then the continuous mapping theorem implies that $R(Y^n,\theta) \to \exp\{-\frac12 \chisq(D)\}$ in distribution.  

From this, we can conclude $\| G_n^\theta - G \|_\infty \to 0$ as $n \to \infty$.  But while the above analysis treated $\theta$ as fixed, all the steps apply locally uniformly in $\theta$: the $M_n$ does not depend on $\theta$, $\Delta_\theta(Y^n)$ converges to Gaussian locally uniformly in $\theta$ as in \eqref{eq:delta}, and the mapping $r \mapsto e^{-r/2}$ is uniformly continuous for $r \geq 0$.  We can, therefore, conclude that $\| G_n^\theta - G \|_\infty \to 0$ locally uniformly in $\theta$, which proves the lemma.
\end{proof}

\begin{lem}
\label{lem:step2}
Let $z_n$ be any sequence that is either bounded or is diverging to $\infty$ but no faster than $n^{1/2}$, and then set $\theta_n^{z_n} = \Theta + n^{-1/2} z_n$; in either case, $\theta_n^{z_n}$ is bounded.  Then, under the stated regularity conditions,  
\[ G\{ R(X^n, \theta_n^{z_n}) \} - \gamma_{X^n}(\theta_n^{z_n}) \to 0 \quad \text{in $\prob_\Theta$-probability, $n \to \infty$}. \]
The same conclusion holds even if $z_n$ is a sequence of random vectors having the aforementioned properties with $\prob_\Theta$-probability tending to 1.
\end{lem}

\begin{proof}
We will start with the case of bounded $z_n$.  For now, take a fixed $z$ and take $\theta_n^z = \Theta + n^{-1/2} z$.  Then consider a slightly modified log-likelihood ratio process, one where $\theta_n^z$ is the anchor and departures from the anchor are indexed by $u \in \RR^D$:
\begin{align*}
2\log \frac{L_{X^n}(\theta_n^z + n^{-1/2} u)}{L_{X^n}(\theta_n^z)} & = 2\log \frac{L_{X^n}(\Theta + n^{-1/2} (z+u))}{L_{X^n}(\Theta)} - 2\log \frac{L_{X^n}(\Theta + n^{-1/2} z)}{L_{X^n}(\Theta)} \\
& = 2u^\top I_\Theta \, \Delta_\Theta(X^n) - (u^\top I_\Theta \, u - 2 z^\top I_\Theta \, u) + o_{\prob_\Theta}(1).
\end{align*}
The last line follows by applying the local asymptotic normality expansion \eqref{eq:lan} to both terms in the previous line.  Just like in the proof of Lemma~\ref{lem:step1}, the approximation can be made uniform, so that the supremum of the left-hand side---which is equal to $-2\log R(X^n, \theta_n^z)$---can be asymptotically approximated in $\prob_\Theta$-probability by the supremum of the right-hand side.  Here, too, the supremum of the right-hand side has a closed-form expression, so the above discussion implies 
\begin{equation}
\label{eq:wilks1}
-2\log R(X^n, \theta_n^z) - \{z - \Delta_\Theta(X^n)\}^\top I_\Theta \{z - \Delta_\Theta(X^n)\} \to 0, \quad \text{in $\prob_\Theta$-probability}. 
\end{equation}
It follows from, e.g., Proposition~2.1.2 in \citet{bickel1998} that the above convergence also holds locally uniformly in $z$.  Define the function 
\[ f(y) = G(e^{-\frac12 y}), \quad y \geq 0, \]
so that the following two identities hold: 
\begin{align*}
G\{ R(X^n, \theta_n^z) \} & = f\{ -2\log R(X^n, \theta_n^z) \} \\
\gamma_{X^n}(\theta_n^{z}) & = f\bigl[ \{z - \Delta_\Theta(X^n)\}^\top I_\Theta \{z - \Delta_\Theta(X^n)\} \bigr].
\end{align*}
Then the difference between the two left-hand sides is 
\begin{align*}
G\{ R(X^n, & \, \theta_n^z) \} - \gamma_{X^n}(\theta_n^{z}) \\
& = f\{ -2\log R(X^n, \theta_n^z) \} - f\bigl[ \{z - \Delta_\Theta(X^n)\}^\top I_\Theta \, \{z - \Delta_\Theta(X^n)\} \bigr] 
\end{align*}
and, since $f$ is continuous, it follows from \eqref{eq:wilks1} and the continuous mapping theorem that the above difference converges in $\prob_\Theta$-probability to 0 as $n \to \infty$.  It is easy to check that the derivative of $f$ is equal to the negative of the $\chisq(D)$ density function and, therefore, is uniformly bounded in magnitude.  This implies that $f$ is uniformly continuous and, therefore, the locally uniform convergence in $z$, from \eqref{eq:wilks1}, carries over the last display.  Therefore, 
\[ G\{ R(X^n, \theta_n^z) \} - \gamma_{X^n}(\theta_n^{z}) \quad \text{in $\prob_\Theta$-probability, locally uniformly in $z$}, \]
which, of course, implies the lemma's claim for a bounded sequence $z_n$ in place of the fixed $z$ above.  The same holds for a sequence $z_n$ of random vectors that is bounded in $\prob_\Theta$-probability since the intersection of two sets with $\prob_\Theta$-probability tending to 1 also has $\prob_\Theta$-probability tending to 1.  

Next, we investigate the case where $\theta_n^{z_n}$ is bounded but $z_n$ is divergent subject to the constraint $z_n \lesssim n^{1/2}$.  Unlike in the previous case, here $\prob_{\theta_n^{z_n}}$ and $\prob_\Theta$ are not contiguous, so certain likelihood ratios will have extreme limits.  We proceed to prove the lemma's claim by showing that both $G\{ R(X^n, \theta_n^{z_n}) \}$ and $\gamma_{X^n}(\theta_n^{z_n})$ are vanishing $\prob_\Theta$-probability in this case.  First, consider $G\{ R(X^n, \theta)\}$ for a fixed but generic $\theta$ in a bounded neighborhood of $\Theta$.  It is easy to see that 
\[ \log R(X^n, \theta) \leq \ell_{X^n}(\theta) -  \ell_{X^n}(\Theta), \]
and the upper bound is a sum of iid random variables, which we will write as $\widehat\prob_n d_\theta$, with $\widehat\prob_n$ the empirical distribution based on $X^n$ and $d_\theta(x) = \log p_\theta(x) - \log p_\Theta(x)$.  The expected value of $d_\theta(X)$ is $\prob f_\theta = -K(p_\Theta, p_\theta)$, the negative Kullback--Leibler divergence of $\prob_\theta$ from $\prob_\Theta$.  The variance of $d_\theta(X)$ will be denoted by $v(\theta)$, and, by the theorem's Lipschitz condition, is a bounded function of $\theta$.  Then 
\[ \log R(X^n, \theta) \leq \ell_{X^n}(\theta) -  \ell_{X^n}(\Theta) = n^{1/2} \, \widehat{\mathsf{G}}_n d_\theta - n K(p_\Theta, p_\theta), \]
where $\widehat{\mathsf{G}}_n = n^{1/2}(\widehat{\prob}_n - \prob)$ is the empirical process.  Similar to the argument presented in \citet{gombay1997}, the central limit theorem applies to the empirical process term for fixed $\theta$, so that $\widehat{\mathsf{G}}_n d_\theta$ has a Gaussian limit and, in particular, is $O_{\prob_\Theta}(1)$ as $n \to \infty$.  The Lipschitz condition mentioned above implies that the collection of mappings $f_\theta$ indexed by $\theta$ in a bounded neighborhood of $\Theta$ is Donsker \citep[][Example~19.7]{vaart1998} and, therefore, $\theta \mapsto \widehat{\mathsf{G}}_n d_\theta$ converges in distribution to a Gaussian process, uniformly in $\theta$.  Then the above claim ``$\widehat{\mathsf{G}}_n d_\theta = O_{\prob_\Theta}(1)$ as $n \to \infty$'' holds uniformly in $\theta$.  Putting everything together, and replacing the generic bounded $\theta$ with $\theta_n^{z_n}$, we have 
\[ \log R(X^n, \theta_n^{z_n}) \leq \{ n v(\theta_n^{z_n}) \}^{1/2} \, v(\theta_n^{z_n})^{-1/2} \, \widehat{\mathsf{G}}_n d_{\theta_n^{z_n}} - n K(p_\Theta, p_{\theta_n^{z_n}}). \]
The scaled empirical process $v(\theta_n^{z_n})^{-1/2} \, \widehat{\mathsf{G}}_n d_{\theta_n^{z_n}}$ is $O_{\prob_\Theta}(1)$ uniformly in $z_n$ that ensure $\theta_n^{z_n}$ is bounded; the scaling forces the marginal distributions of the Gaussian process to be standard normal.  Again, the theorem's Lipschitz condition implies that 
\[ K(p_\theta, p_{\theta_n^{z_n}}) \lesssim n^{-1} z_n^2 \quad \text{and} \quad v(\theta_n^{z_n}) \lesssim n^{-1} z_n^2. \]
Therefore, since $z_n \to \infty$, we conclude that 
\begin{align*}
\log R(X^n, \theta_n^{z_n}) \leq z_n \, O_{\prob_\Theta}(1) - z_n^2 \to -\infty & \implies R(X^n, \theta_n^{z_n}) \to 0 \\
& \implies G\{ R(X^n, \theta_n^{z_n}) \} \to 0,
\end{align*}
and this conclusion holds in $\prob_\Theta$-probability for all (random) $z_n$ sequences such that $\theta_n^{z_n}$ remains bounded.  This proves the first of our two-part claim. 

For the second part, consider $\gamma_{X^n}(\theta_n^{z_n})$.  In the first part of the proof, recall that 
\[ \gamma_{X^n}(\theta_n^{z_n}) = f\bigl[ \{z_n - \Delta_\Theta(X^n)\}^\top I_\Theta \, \{z_n - \Delta_\Theta(X^n)\} \bigr], \]
for the continuous function $f(y) = G(e^{-y/2})$ defined above.  Since $\Delta_\Theta(X^n)$ is bounded in $\prob_\Theta$-probability, it is clear that the right-hand side above converges in $\prob_\Theta$-probability to $f(\infty) = G(0) = 0$ for any (random) sequence $z_n$ that diverges (with $\prob_\Theta$-probability tending to 1).  We have now taken care of the ``$z_n$ bounded'' and ``$z_n$ divergent but $\lesssim n^{1/2}$'' cases, for both deterministic and random $z_n$, so the proof is complete. 
\end{proof}

\subsection{Proof of Theorem~\ref{thm:bvm.pr}: sketch}
\label{SS:bvn.pr.proof}

Since we already presented a lengthy, thorough proof of Theorem~\ref{thm:bvm}, we only give a (semi-detailed) sketch of the proof of Theorem~\ref{thm:bvm.pr} here.  Just like in the proof of Theorem~\ref{thm:bvm}, we start with the following decomposition: 
\[ \bigl| \pi_{X^n}^\text{\sc pr}(\phi) - \gamma_{X^n}^\text{\sc pr}(\phi) \bigr| 
\leq \Bigl\| \sup_{\lambda \in \LL_0} G_n^{\phi,\lambda} - G \Bigr\|_\infty + \bigl| G\{ R^\text{\sc pr}(X^n, \phi) \} - \gamma_{X^n}^\text{\sc pr}(\phi) \bigr|, \]
where $G_n^{\phi,\lambda}$ and $G$ are as defined previously.  When evaluated at $\phi=\phi_n^z$, the upper bound in the above display satisfies 
\[ \sup_{z \in K} \Bigl\| \sup_{\lambda \in \LL_0} G_n^{\phi_n^z,\lambda} - G \Bigr\|_\infty \leq \sup_{z \in K, \lambda \in \LL_0} \|G_n^{\phi_n^z, \lambda} - G \|_\infty \leq \sup_{(\phi,\lambda) \in \text{compact}} \| G_n^{\phi,\lambda} - G \|_\infty. \]
It turns out that the profile relative likelihood has asymptotic properties very similar to those of the profile likelihood as fleshed out in the proof of Theorem~\ref{thm:bvm} above.  In particular, as presented in, e.g., Equation~(6) of \citet{murphy.vaart.2000}, but with different notation, the profile relative likelihood satisfies 
\begin{equation}
\label{eq:bound.prl}
-2\log R^\text{\sc pr}(X^n, \phi_n^z) = \{ z - \widetilde\Delta_{\Phi,\Lambda}(X^n)\}^\top (n\tilde I_{\Phi,\Lambda}) \, \{ z - \widetilde\Delta_{\Phi,\Lambda}(X^n)\} + o_{\prob_{\Phi,\Lambda}}(1), 
\end{equation}
where the latter error term is uniform for $z$ in a compact set $K$.  So, the same argument made in the proof of Theorem~\ref{thm:bvm} applies here too, so we conclude that 
\[ \sup_{z \in K} \Bigl\| \sup_{\lambda \in \LL_0} G_n^{\phi_n^z,\lambda} - G \Bigr\|_\infty \to 0, \quad \text{in $\prob_{\Phi,\Lambda}$-probability}. \]
By the property \eqref{eq:bound.prl} and the definition of $\gamma_{X^n}^\text{\sc pr}$ in \eqref{eq:mpl.pr.lim}, we clearly also get 
\[ \sup_{z \in K} \bigl| G\{ R^\text{\sc pr}(X^n, \phi_n^z) \} - \gamma_{X^n}^\text{\sc pr}(\phi_n^z) \bigr| \to 0, \quad \text{in $\prob_{\Phi,\Lambda}$-probability}. \]
This completes the (sketch of the) proof of Theorem~\ref{thm:bvm.pr}.

\subsection{Proof of Theorem~\ref{thm:bvm.ex}}
\label{SS:bvn.ex.proof}

The proof is pretty straightforward.  The only thing that $\pi_{X^n}^\text{\sc ex}(\phi)$ could possibly merge with is the supremum of the limiting version of the original contour, i.e., $\gamma_{X^n}(\phi,\lambda)$, over the nuisance parameter $\lambda$.  The uniformity established in Theorem~\ref{thm:bvm} makes this argument rigorous.  Now it is just a matter of carrying out the optimization in the limiting Gaussian contour $\gamma_{X^n}$.  But we did this already in Section~\ref{SS:possibility}---see Equation~\eqref{eq:gauss.extension}---when discussing the extension principle applied to Gaussian possibility measures.  Indeed, that same argument applied to the particular Gaussian here gives a contour of the form \eqref{eq:mpl.ex.lim}, which proves the claim.

\bibliographystyle{apalike}
\bibliography{mybib}

\end{document}